\documentclass[a4paper,10.5pt]{amsart}%
\usepackage{amsfonts}
\usepackage{amssymb}
\usepackage{amssymb}
\usepackage{amscd}
\usepackage[dvips]{graphicx}
\usepackage{amsmath}%
\providecommand{\U}[1]{\protect\rule{.1in}{.1in}}
\newtheorem{lemma}{Lemma}
\newtheorem{corollary}{Corollary}
\newtheorem{proposition}{Proposition}
\newtheorem{theorem}{Theorem}

\def\eqn {\begin{equation}}
\def\eeqn {\end{equation}}

\input epsf
\begin{document}
\title{Large Time Behavior of Exchange-Driven Growth}
\date{April 2019}
\author{Emre Esent\"{u}rk}
\address{University of Warwick, Mathematics Institute, UK}
\email{E.esenturk.1@warwick.ac.uk}
\author{Juan Velazquez}
\address{Hausdorff Centre for Mathematics}
\email{velazquez@iam.uni-bonn.de}
\thanks{Corresponding author email: E.esenturk.1@warwick.ac.uk}

\begin{abstract}
Exchange-driven growth (EDG) is a process in which pairs of clusters interact
by exchanging single unit with a rate given by a kernel $K(j,k)$. Despite EDG
model's common use in the applied sciences, its rigorous mathematical
treatment is very recent. In this article we study the large time behaviour of
EDG equations. We show two sets of results depending on the properties of the
kernel $(i)$ $K(j,k)=b_{j}a_{k}$ and $(ii)$ $K(j,k)=ja_{k}+b_{j}%
+\varepsilon\beta_{j}\alpha_{k}$. For type I kernels, under the detailed
balance assumption, we show that the system admits equilibrium solutions up to
a critical mass $\rho_{s}$ above which there is no equilibrium. We prove that
if the system has an initial mass above $\rho_{s}$ then the solutions converge
to critical equilibrium distribution in a weak sense while strong convergence
can be shown when initial mass is below $\rho_{s}$. For type II kernels, we
make no assumption of detailed balance and equilibrium is obtained via a
contraction property. We provide two separate results depending on the
monotonicity of the kernel or smallness of the total mass. For the first case
we show exponential convergence in the number of clusters norm and for the
second we prove exponential convergence in the total mass norm.

\end{abstract}
\maketitle

\section{Introduction}

Exchange-driven growth (EDG) is a model for non-equilibrium cluster growth in
which pairs of clusters interact by exchanging a single unit of mass (monomer)
\cite{Naim},\cite{Naim-book}. In the recent years EDG has been used to model
several natural and social phenomena such as migration \cite{Ke2}, population
dynamics \cite{Leyvraz2} and wealth exchange \cite{Isp}. EDG is also important
mathematically for multiple reasons. Firstly, it is a model of intermediate
complexity between the classical Becker-Doring (BD) model \cite{Wattis},
\cite{Barb}, where the dynamics are well understood, and the Smoluchowski
coagulation model, where the existing mathematical questions are much tougher.
Secondly, EDG arises as the mean field limit of a class of interacting
particle systems (IPS) that includes models of non-equilibrium statistical
physics including zero-range processes \cite{Stef1}, \cite{Stef2},
\cite{Godrec}, \cite{Stef3}, \cite{Godrec2}, \cite{Beltran}, \cite{Stef4},
\cite{Waclaw}, \cite{Colm}, that have been intensively studied for a range of
condensation phenomena that they exhibit. Despite its importance, rigorous
results on the properties and behavior of the corresponding equations
(existence, uniqueness, asymptotic behavior etc.) are scarce and have been
obtained only very recently \cite{EE}, \cite{Sch}. It is the purpose of this
article to continue the mathematical study of the EDG systems focusing on the
large time asymptotic properties of solutions with explicit convergence rates
where possible.

In EDG, the mathematical description of the mass exchange systems is given at
the mesoscopic level and one studies the mean field rate equations (hereafter
referred as EDG equations) ignoring fluctuations at the particle level. The
main mathematical object of study is $c_{j}(t)$, the cluster size
distribution, describing the volume fraction of the system which is occupied
by clusters of size $j\geq1$, where $j=0$ corresponds to the empty (available)
volume fraction not occupied by any cluster. Symbolically, the exchange
process can be described in the following way. If $<j>,$ $<k>$ denote the
non-zero clusters of sizes $j,k>0,$ then the rule of interaction is
\[
<j>\oplus<k>\ \rightarrow\ <j\pm1>\oplus<k\mp1>.
\]
If, one of the clusters is a zero-cluster ($0$-cluster)$,$ then the rule is
given by
\[
<j>\oplus<0>\ \rightarrow\ <j-1>\oplus<1>.
\]

The rate of exchange from a $j-$cluster to a $k-$cluster is given by $K(j,k)$
which is not necessarily symmetric$.$ This is an important difference between
the EDG and coagulation (Smoluchowski) models. Mathematically, the infinite
network of interactions are represented as system of nonlinear ODEs%

\begin{equation}
\dot{c}_{0}=c_{1}\sum_{k=0}^{\infty}K(1,k)c_{k}-c_{0}\sum_{k=1}^{\infty
}K(k,0)c_{k}\text{,} \label{0-infode}%
\end{equation}

\begin{align}
\text{ }\dot{c}_{j}  &  =c_{j+1}\sum_{k=0}^{\infty}K(j+1,k)c_{k}-c_{j}%
\sum_{k=0}^{\infty}K(j,k)c_{k}\label{infode}\\
&  -c_{j}\sum_{k=1}^{\infty}K(k,j)c_{k}+c_{j-1}\sum_{k=1}^{\infty
}K(k,j-1)c_{k}\text{ ,\ }%
\end{align}%
\begin{equation}
c_{j}(0)=c_{j,0}\text{ \ \ \ }\{j=0,1,2,...\}. \label{infIC}%
\end{equation}
\bigskip

In \cite{EE} one of the authors provided the first mathematical investigation
of EDG equations giving the fundamental properties such as global existence,
uniqueness and non-existence. In particular, for general non-symmetric kernels
whose growth is bounded as $K(j,k)\leq Cjk$ (for large $j,k),$ unique
classical solutions were shown to exist globally. For symmetric kernels, it
was shown that the existence result can be generalized to kernels whose growth
rate is lying in the range $K(j,k)\leq C(j^{\mu}k^{v}+j^{\nu}k^{\mu}),$ with
$\mu,\nu\leq2,$ $\mu+v\leq3$, a fact was first discovered by physicists based
on scaling arguments \cite{Naim}. Uniqueness of solutions was obtained under
additional boundedness assumptions on moments. Recently this result was
extended in \cite{Sch} without requiring the moment assumption. On the other
hand, for fast growing kernels it was shown that the solutions cannot exist
provided that the initial distribution has sufficiently fat tails.

There exists a body of literature for applications of EDG mechanism in the
physical and social sciences. In these classical treatments exchange
interactions are only defined among non-zero clusters and $0$-clusters have no
use or meaning. One of the key aspects of the current formulation of the EDG
system given by (\ref{0-infode})-(\ref{infIC}) is the inclusion of the
$0-$clusters (or available volume) representing the non-zero volume fraction
accessible to particles. Hence, in this description total volume density,
i.e., $\sum_{j\geq0}c_{j}=\eta$ becomes a conserved quantity independently of
the total mass density (denoted by $\rho$ hereafter).

The presence of accessible (available) volume influence the properties of the
whole system most distinctly by allowing the particles to detach from non-zero
clusters and re-occupy the available (free) volume which mathematically reads
as $K(j,0)>0$. Effectively, this provides a "fresh" source of $1-$clusters to
the system. This behaviour was first demonstrated numerically in \cite{EE2},
where it was observed that the seemingly innocuous change in the kernel
($K(j,0)>0)$ fundamentally alters the dynamical behavior, driving the system,
towards a unique equilibrium (BD-like) instead of indefinite growth where the
cluster densities eventually vanish (Smoluchowski-like when $K(j,0)=0$). For a
class of kernels this observation was recenty proven in \cite{Sch}.

In this article we study the large time behavior of the exchange-driven system
concentrating on the cases where the exchange interaction rate (i.e., the
kernel $K)$ is separable as follows
\[
(I)\text{ }K(j,k)=b_{j}a_{k}%
\]%
\[
(II)\ K(j,k)=ja_{k}+b_{j}+\varepsilon\beta_{j}\alpha_{k}%
\]
where the $b_{j}$ (and $\beta_{j})$ terms can be interpreted as "export" rate
and $a_{j}\,($and$\ \alpha_{j})$ terms as the "import" rate of particles from
a cluster and $II,$ $\varepsilon>0$ is a small parameter.

For the type $I$ separable kernels we show that, under a crucial balance
assumption (of density fluxes), the equilibrium cluster densities take the
form $c_{j}=\frac{Q_{j}z(\rho,\eta)^{j}}{\sum_{j}Q_{j}z(\rho,\eta)^{j}}$ where
$z(\rho,\eta)$ is a solution of a nonlinear equation and $Q_{j}=%
{\textstyle\prod_{k=1}^{k=j}}
\frac{a_{k-1}}{b_{k}}$ are combinatorial factors. The explicit form of the
equilibria becomes important as it serves useful in the analysis of behavior
of solutions. In particular, the feature that the equilibrium densities are
the minimizers of a certain functional (entropy) $V(c)=\sum c_{j}\ln
(\frac{c_{j}}{Q_{j}})-c_{j}$ on a chosen set
\[
X_{\rho,\eta}=\{(c)_{j=1}^{\infty}:c_{j}\geq0,\sum jc_{j}=\rho,\sum c_{j}%
=\eta\}
\]
enable us to use the well developed entropy dissipation methods for the large
time analysis. It is worth noting that for this type of kernel, equilibrium
solution is possible only for a finite range of initial mass $\rho_{i}$
satisfying $\rho_{i}\leq\rho_{c}$ (hereafter referred to as subcritical case)
where individual cluster densities can be explicitly obtained from a recursive
relation. If the $\rho_{i}>\rho_{c}$ there will be no equilibrium solutions
indicating a phase transition. For type $II$ separable kernels we do not make
any assumption on the structure of equilibrium (no detailed balance
assumption) and therefore no specific analysis of the forms will be made or
needed except for the existence of equilibrium. That we do not impose any
structural conditions on the equilibrium is one of the novelties in this paper.

The main goal of this article is to obtain rigorous results on the large time
behavior of the EDG system. Below we give a brief outline of arguments and
main findings. We provide two sets of results depending on the type of the kernel.

For type $I$ kernels, we prove qualitative convergence results with mild
assumptions on the kernel. In particular, we show that the time dependent
system (\ref{0-infode})-(\ref{infIC}) goes strongly to equilibrium if the
total mass is below a threshold value $\rho_{c}.$ Above this critical value, a
dynamic phase transition occurs and the excess initial mass $\rho_{i}-\rho
_{c}$ forms larger and larger clusters while the rest of the system approaches
to equilibrium weakly. This behavior is analogous to the simpler Becker-Doring
system whose dynamics has been well studied \cite{Ball}, \cite{Slem},
\cite{Carr}, \cite{Phil}, \cite{Fournier}, \cite{Niet}, \cite{Can1},
\cite{Pego}.

For the results, we first show that the under the assumptions of \cite{EE}
system (\ref{0-infode})-(\ref{infIC}) form a semi-group. Then one naturally
seeks a Lyapunov function which is decreasing in time and a suitable norm
where the positive orbit is relatively compact and the Lyapunov function is
continuous. Since mass is an invariant of the motion a first candidate for the
suitable norm is the space $X=\{(c)_{j=1}^{\infty}:\sum jc_{j}<\infty\}.$ The
downside of this natural norm is that the positive orbit is not always
compact. Quite similar to the classical case in BD equations using a weaker
topology comes useful and the desired compactness result can be obtained even
for the supercritical case. The remaining condition is then to satisfy the
continuity of the Lyapunov function in the chosen metric. It turns out that it
is not generally true for the "bare" form of the Lyapunov function but holds
for modified version%
\[
V_{z,y}(c)=V(c)-\ln z\sum jc_{j}-\ln y\sum c_{j}%
\]
Here, the invariance of the total mass and volume is of crucial importance for
preserving the monotonicity property of the new Lyapunov function. This
naturally extends the approach taken in \cite{Ball} where the only conserved
quantity was total mass. With this modification we can show that $V_{z,y}$ is
weak (defined more precisely later) continuous and the invariance principle
can be applied to prove the weak convergence of solutions. For the subcritical
case we enforce stronger conditions on the initial data to prove compactness
and use the invariance principle to show the strong convergence.

Our second set of results with type $II$ kernels on the large time behavior
concern the convergence to equilibrium solutions without detailed balance.
Both the existence of general equilibrium and the convergence to equilibrium
is a consequence of key contraction properties of solutions. We present two
different results of convergence depending on the assumptions on the
$a_{j},b_{j}$ functions. For each result we show that solutions converge to
the equilibrium exponentially fast.

The proof of rate of convergence relies on analyzing the evolution of two
non-negative quantities which measure the distance of a solution from another
solution (distribution) having the same mass. One then shows that this
"distance" shrinks in time (contraction property). To show the first
contraction property we assume the kernel satisfies certain monotonicity
conditions. With this, one can show that solutions approach to equilibrium
exponentially fast in the "weak" norm (same as in type $I$ kernels).
Alternatively, one can remove the monotonicity conditions on the kernel and
impose a small mass condition on the system. The second approach is along the
lines of \cite{Fournier}. Though more restrictive, with such an assumption one
can show that solutions converge to equilibrium exponentially fast in the
strong topology (same as in type $I$ kernels$)$.

Part of the results of this paper, namely those in Section 3, overlap with
some of the results in \cite{Sch} which were independently obtained. Actually
the results in \cite{Sch} cover a class of kernels wider than those considered
in Section 3 of this paper. Nevertheless, given that the proofs of the
convergence results are simpler and give a clear intuition about properties of
the kernels for the product kernels considered in Section 3 we decided to keep
them (see the discussion about "export" and "import" tendencies). On the other
hand, the analysis of the long time asymptotics for kernels of type $II$, for
which detailed balance is not satisfied, has not been considered to our
knowledge anywhere else. We consider this type of kernels in Section 4 of this
paper. In addition to providing the first in providing explicit rates of
convergence, the results in this article are also relevant as they illustrate
that the EDG system shows structural similarities to the BD system and
naturally generalizes it.

The organization of the rest of the paper is as follows. In Section 2 we
recall some of the basic results on the well posedness of the EDG system and
give important lemmas that will be used throughout. In Section 3, we study the
form of the equilibria with type $I$ kernels and define and analyze some
important functions that will form the basis of arguments to prove the
convergence to equilibrium (in weak and strong senses). In Section 4, we study
the EDG system with type $II$ kernels without the detailed balance assumption
and prove exponential convergence to equilibrium in weak and strong senses
with explicit rates.

\section{FUNDAMENTALS}

In this section we give the setting of the problem and provide some basic
facts which will be used in the subsequent analysis. For the sequence of
functions we are concerned the appropriate spaces are $X_{\mu}=\{x=(x_{j}),$
$x_{j}\in\mathbb{R};\left\Vert x\right\Vert _{\mu}<\infty\}.$ We equip the
space with the norm $\left\Vert x\right\Vert _{\mu}=\sum_{j=1}^{\infty}j^{\mu
}x_{j}$ where $\mu\geq0.$ Also, the cluster interaction kernel $K(\cdot
,\cdot):\mathbb{R}\times\mathbb{R}\rightarrow\lbrack0,\infty)$ is defined to
be non-negative throughout and set $K(0,j)\equiv0$ identically.

\textbf{Definition 1: }We say the system\ (\ref{0-infode})-(\ref{infIC}) has a
solution iff

$(i)$ $c_{j}(t)$ $:[0,\infty)\rightarrow\lbrack0,\infty)$ is continuous and
$\sup_{t\in\lbrack0,\infty)}c_{j}(t)<\infty$

$(ii)$ $\int_{0}^{t}\sum_{k=0}^{\infty}K(j,k)c_{k}ds<\infty,$ $\int_{0}%
^{t}\sum_{k=1}^{\infty}K(k,j)c_{k}ds<\infty$ for all $t\in\lbrack0,T)$
($T\leq\infty)$

$(iii)\ c_{j}(t)=c_{j}(0)+\int_{0}^{t}\left(  c_{j+1}\sum_{k=0}^{\infty
}K(j+1,k)c_{k}-c_{j}\sum_{k=0}^{\infty}K(j,k)c_{k}\right)  ds$

$\ \ \ \ \ \ \ \ \ \ \ \ \ \ \ \ \ \ \ \ \ \ \ \ \ \ \ \ +\int_{0}^{t}\left(
-c_{j}\sum_{k=1}^{\infty}K(k,j)c_{k}+c_{j-1}\sum_{k=1}^{\infty}K(k,j-1)c_{k}%
\right)  ds$ \ $\{j>0\}$

\ \ \ \ \ \ $c_{0}(t)=c_{0}(0)+\int_{0}^{t}c_{1}\sum_{k=0}^{\infty}%
K(1,k)c_{k}-c_{0}\sum_{k=1}^{\infty}K(k,0)c_{k}.$

\textbf{Definition 2:} For a sequence $(c_{j})_{j=1}^{N}$, we call the
quantity $M_{p}^{N}(t)=\sum_{j=0}^{N}j^{p}c_{j}(t)$ as the $p^{th}-$moment of
the sequence. If the sequence is infinite, then we denote the $p^{th}-$moment
with $M_{p}(t)=\sum_{j=0}^{\infty}j^{p}c_{j}(t).$

\bigskip It is often useful to study the finite version of the infinite system
where the equations are truncated at some order, say, $N<\infty$ as below
\begin{equation}
\text{\ }\dot{c}_{0}^{N}=c_{1}^{N}\sum_{k=0}^{N-1}K(1,k)c_{k}^{N}-c_{0}%
^{N}\sum_{k=1}^{N}K(k,0)c_{k}^{N}, \label{Tode0}%
\end{equation}

\begin{align}
\text{\ }\dot{c}_{j}^{N}  &  =c_{j+1}^{N}\sum_{k=0}^{N-1}K(j+1,k)c_{k}%
^{N}-c_{j}^{N}\sum_{k=0}^{N-1}K(j,k)c_{k}^{N}\label{Tode-j}\\
&  -c_{j}^{N}\sum_{k=1}^{N}K(k,j)c_{k}^{N}+c_{j-1}^{N}\sum_{k=1}%
^{N}K(k,j-1)c_{k}^{N},\text{ }\{1\leq j\leq N-1\}\text{\ \ }\nonumber
\end{align}%
\begin{equation}
\dot{c}_{N}^{N}=-c_{N}^{N}\sum_{k=0}^{N-1}K(N,k)c_{k}^{N}+c_{N-1}^{N}%
\sum_{k=1}^{N}K(k,N-1)c_{k}^{N}, \label{TodeN}%
\end{equation}
with the initial conditions given by%
\begin{equation}
c_{j}^{N}(0)=c_{j,0}\geq0,\text{ \ }\{0\leq j\leq N\}. \label{TodeIC}%
\end{equation}

The fundamental properties of solutions are well known from the standard ODE
theory. We also quote the following basic result from \cite{EE} whose proof we skip

\begin{lemma}
\bigskip Let $g_{j}$ be a sequence of non-negative real numbers. Then,
\begin{equation}
\sum_{j=0}^{N}g_{j}\frac{dc_{j}}{dt}=\sum_{j=1}^{N}(g_{j-1}-g_{j})c_{j}%
^{N}\sum_{k=0}^{N-1}K(j,k)c_{k}^{N}+\sum_{j=0}^{N-1}(-g_{j}+g_{j+1})c_{j}%
^{N}\sum_{k=1}^{N}K(k,j)c_{k}^{N}. \label{mom-red1}%
\end{equation}

\end{lemma}

\bigskip Two immediate results that one can draw from this lemma (by setting
$g_{j}=1$ or $g_{j}=j)$ is the conservation of total volume and total mass
which will also extend to the infinite system. The finite system will be
revisited to obtain estimates on the solutions where the original system can
pose subtleties.

Now, we state the some of the fundamental results on the solutions of the EDG
system (\ref{0-infode})-(\ref{infIC}) with kernels allowing particles to hop
on to the available volume ($K(j,0)>0),$ sometimes called as non-linear
chipping. At this point, no assumptions are made on the kernel, but we always
assume the growth rate of the kernels to be sublinear (see \cite{EE} for
well-posedness results for kernels growing faster than linear).

\begin{theorem}
\label{main-exist}Let $K(j,k)$ be a general kernel satisfying $K(j,k)\leq Cjk$
for large enough $j,k.$ Assume further that $M_{p}(0)<\infty$ for some $p>1$.
Then the infinite system (\ref{0-infode})-(\ref{infIC}) has a global solution
$(c_{j})\in X_{1}$ where $c_{j}(t)$ is continuously differentiable. Moreover
$M_{p}(t)<\infty$ and for all $t<\infty$ and%
\begin{align}
\sum_{0}^{\infty}c_{j}(t)  &  =\sum_{0}^{\infty}c_{j}(0),\label{vol-cons}\\
\sum_{0}^{\infty}jc_{j}(t)  &  =\sum_{0}^{\infty}jc_{j}(0). \label{mass-cons}%
\end{align}

\end{theorem}

We note that the global existence and conservation laws still hold if one
replaces the moment assumption ($M_{p}(0)<\infty)$ with a slower growth
assumption on the kernels.

\begin{theorem}
\label{main-exist2}Let $K$ satisfy $K(j,k)\leq Cb_{j}a_{k}$ $($with
$a_{j},b_{j}=o(r)).$Then the infinite system (\ref{0-infode})-(\ref{infIC})
has a global solution $(c_{j})\in X_{1}$ where $c_{j}(t)$ is continuously differentiable.
\end{theorem}

While Theorem \ref{main-exist} shows that individual cluster size densities
are continuous in time, when studying the asymptotics we will need to work
with the cluster size distribution as an element of the space $X_{1}.$ The
following result, which is an immediate consequence of Dini's uniform
convergence theorem, gives the continuity of $c(t).$

\begin{proposition}
\label{contX}Let $c$ be the solution of (\ref{0-infode})-(\ref{infIC}). Then
$c:[0,T)\rightarrow X$ is continuous.
\end{proposition}

When discussing the convergence to equilibrium for the super-critical case, in
addition to strong convergence (in the $X_{1}$ norm) we will also use the
concept of weak$\ast$ convergence which has also been frequently used in the
analysis of the Becker-Doring equations.

\textbf{Definition 3: }We say that a sequence $\{x^{i}\}$ in $X_{1}$ converges
weak$\ast$ to $x\in X_{1}$ ($\rightharpoonup^{\ast}$symbolically$)$ if the
following holds

$(i)$ $\sup\left\Vert x^{i}\right\Vert <\infty$

$(ii)$ $x_{j}^{i}\rightarrow x_{j}$ as $i\rightarrow\infty$ for each
$j=1,2,...$

The virtue behind using this concept of convergence is two-fold. First, as
briefly mentioned in the introduction, the positive orbit of the flow
generated by EDG\ equations are not generally compact in $X_{1}.$ In those
cases it will be convenient to consider a finite ball for the flow $B_{\rho
}=\{x\in X_{1},\left\Vert x\right\Vert <\rho\}$ induced with the metric
\[
dist(x,y)=\sum_{j=0}^{\infty}\left\vert x_{j}-y_{j}\right\vert
\]
where the $B_{\rho}$ is compact and the weak$\ast$ convergence is equivalent
to convergence in this new metric$.$A second benefit of studying the
weak$\ast$ convergence is that one can easily characterize the cases where
weak convergence becomes equivalent to strong convergence in $X_{1}$ thanks to
the following lemma \cite{Ball}.

\begin{lemma}
\label{mag2norm} If $x^{j}\rightharpoonup^{\ast}x$ in $X_{1}$ and $\left\Vert
x^{j}\right\Vert \rightarrow\left\Vert x\right\Vert ,$ then it follows that
$x^{j}\rightarrow x.$
\end{lemma}

In this new topology we make use of a modified concept of continuity which is
defined as below.

\textbf{Definition 4: }Let $S\subset X_{1}.$ A function $f:S\rightarrow R$ is
said to be weak$\ast$ continuous iff $x^{j}\rightharpoonup^{\ast}x$ implies
$f(x^{j})\rightharpoonup^{\ast}f(x)$ as $j\rightarrow\infty.$

A typical example of weak$\ast$ continuous function in $X_{1}$ is the function
$\beta(x)=\sum_{j=0}^{\infty}g_{j}x_{j}.$ This function is weak$\ast$
continuous if and only if the coefficients satisfy $g_{j}=o(j)$ near infinity.

As the last item of this section we relate and establish the link between the
solutions generated by the EDG equations (under the setting of this paper) and
the concept of generalized flow introduced in \cite{Ball} which is defined as below

\textbf{Definition 5:} A generalized flow $G$ on a metric space $Y$ is a
family of continuous mappings $\phi:[0,\infty)\rightarrow Y$ with the properties

$(i)$ if $\phi\in G$ and $t\geq0$ then $\phi_{t}\in G$ and $\phi_{t}%
(s)=\phi(t+s)$

$(ii)$ if $y\in Y$ there exists at least one $\phi\in G$ with $\phi(0)=y$

$(iii)$ if $\phi^{i}\in G$ with $\phi^{i}(0)$ converges to $y~$in $Y$, then
there exists a subsequence $\phi^{i(k)}$ and an element of $\phi\in G$ such
that $\phi^{i(k)}(t)\rightarrow\phi(t)$ uniformly on compact intervals of
$[0,\infty)$ (with $\phi(0)=y).$

The generalized flow is related to semigroup in the following way.

\textbf{Definition 6.} We say the a generalized flow is a semigroup if for
each $y\in Y,$ there is a unique $\phi(t)$ with $\phi(0)=y$ and the flow is
given by a map $T(t):Y\rightarrow Y$ such that $T(t)\phi(0)=\phi(t)$ and $T$
satisfying the properties

$(i)$ $T(0)=identity$

$(ii)$ $T(s+t)=T(s)T(t)$

$(iii)$ the mapping $(t,\phi(0))\rightarrow T(t)\phi(0)$ is continuous from
$[0,\infty)\times Y\rightarrow Y.$

The next results show that EDG system generates a generalized flow in the
strong or weak sense (of continuity) depending on the growth properties of the kernel.

\begin{proposition}
Let the conditions in Theorem \ref{main-exist} hold $(a_{j},b_{j}=O(j)).$ Then
the system (\ref{0-infode})-(\ref{infIC}) generates a generalized flow on
$X^{+.}$
\end{proposition}

\begin{proof}
Properties $(i)$ and $(ii)$ (in Definition 5) are clear from the definition of
a solution of (\ref{0-infode})-(\ref{infIC}). The continuity of $c:[0,\infty
)\rightarrow X^{+}$ is due to Proposition \ref{contX}. For property $(iii),$
consider the sequence of solutions $\phi^{i}(t)$ with initial conditions
$\phi^{i}(0)\rightarrow\phi(0).$ Let $(\phi_{j}^{i})^{N}(t)$ be the
approximation (as in Theorem \ref{main-exist}) of $\phi_{j}^{i}(t)$ such that
$(\phi_{j}^{i})^{N}(t)\rightarrow\phi_{j}^{i}(t)$ as $N\rightarrow\infty$ for
each $i$. Clearly $\phi^{N_{i}}(0)\rightarrow\phi(0)$ as in $X_{1}$ as
$i\rightarrow\infty.$ By the construction in the proof of Theorem
\ref{main-exist} a subsequence (indexed by $N_{i}(k)$ of $\phi_{j}^{N_{i}%
(k)}(t)\rightarrow\phi_{j}(t)$ uniformly for each $j$ for $N_{i}%
(k)\rightarrow\infty$. To show that the convergence is strong in $X^{+}$ we
use the Lemma \ref{mag2norm} and conservation of mass from Theorem
\ref{main-exist}.
\[
\lim_{N_{i}(k)\rightarrow\infty}\sum j\phi_{j}^{N_{i}(k)}(t)=\lim
_{iN_{i}(k)\rightarrow\infty}\sum j\phi_{j}^{N_{i}(k)}(0)=\sum j\phi
_{j}(0)=\sum j\phi_{j}(t)
\]

\end{proof}

\begin{proposition}
Assume the conditions of Theorem \ref{main-exist2} hold $(a_{j},b_{j}=o(j)).$
Then the system (\ref{0-infode})-(\ref{infIC}) generates a generalized flow on
$B_{\rho}^{+}.$
\end{proposition}

\begin{proof}
Consider $d(\phi^{i}(0),\phi(0))\rightarrow0$ in $B_{\rho}^{+}$ and let
$(\phi^{i})^{N}(t)$ be the approximating solutions as in the previous
proposition. From Theorem \ref{main-exist2} a subsequence $\phi_{j}^{N_{i}%
(k)}(t)$ converges uniformly to $\phi_{j}(t)$ uniformly for each $j$ thanks to
the bounds $a_{j},b_{j}=o(j).$ Also, $\{\phi_{j}^{N_{i}(k)}\}$ are uniformly
bounded family. Hence $\phi^{N_{i}(k)}\rightarrow^{\ast}\phi^{i}$ which is
equivalent to $dist(\phi^{N_{i}(k)}(t),\phi(t))\rightarrow0.$
\end{proof}

Since one of the requirements for the generalized flow to be a semigroup is
the uniqueness we need the following uniqueness result from \cite{EE} for the
EDG system.

\begin{theorem}
\label{uniq}Let the conditions of Theorem 1 be satisfied with $M_{2}%
(0)<\infty$. Then the ODE system (\ref{0-infode})-(\ref{infIC}) has a unique
solution in $X_{1}$.
\end{theorem}

With the theorem above and the arguments used in proof of the main existence
theorem one can show that the infinite system (\ref{0-infode})-(\ref{infIC})
actually forms a semigroup.

\begin{theorem}
\label{semiG}Let the conditions of Theorem 1 be satisfied with $M_{2}%
(0)<\infty$. Then the ODE system (\ref{0-infode})-(\ref{infIC}) forms a semigroup.
\end{theorem}

\begin{proof}
Properties $(i)$ and $(ii)$ follow from the definition of solution and Theorem
\ref{main-exist}. Under the conditions of the theorem the uniqueness follows
from Theorem \ref{uniq}. Property $(iii)$ is a consequence of Proposition
\ref{contX}.
\end{proof}

\section{CONVERGENCE TO EQUILIBRIUM WITH DETAILED BALANCE}

\subsection{Equilibria and Minimizers.}

We say that $c_{j}$ is an equilibrium solution if $\dot{c}_{j}(t)=J_{j-1}%
-J_{j}=0$ for all $j\geq0.$ In this section we also impose a structure
(detailed balance) on the equilibrium solutions by setting $J_{-1}=0.$ This
implies $J(j)=0$ for all $j.$ Furthermore, throughout this section we assume
$K(j,k)=b_{j}a_{k}$ (class $I)$ which gives the following recursive
relationship between the cluster densities
\begin{equation}
c_{j}=Q_{j}\frac{B}{A}c_{j-1}=\frac{a_{j-1}...a_{0}}{b_{j}...b_{1}}\left(
\frac{B}{A}\right)  ^{j}c_{0}, \label{equi-form}%
\end{equation}
where $B=\sum_{j=1}^{\infty}b_{j}c_{j}$ and $A=\sum_{k=0}^{\infty}a_{j}c_{j}$
and $\frac{Q_{j}}{Qj-1}=\frac{a_{j-1}}{b_{j}}.$ In order for $c_{j}$ be an
equilibrium state, the set of equations for $c,~A,\ B$ must be solved
simultaneously. We show this by finding a unique distribution given the total
number (density) $\eta$ and total mass (density) $\rho$ of clusters. Consider
the equalities $\eta=\sum_{j=0}^{\infty}Q_{j}\left(  \frac{B}{A}\right)
^{j}c_{0}$ and $\rho=\sum_{j=1}^{\infty}jQ_{j}\left(  \frac{B}{A}\right)
^{j}c_{0},$ we want to show, for given $\rho$ and $\eta,$ there is a unique
$z(\rho,\eta)$ such that $\frac{\sum_{j=1}^{\infty}jQ_{k}z^{j}}{\sum
_{j=0}^{\infty}Q_{j}z^{j}}=\frac{\rho}{\eta}.$

Let $z_{s}$ be the radius of convergence of for the series $\sum_{j=0}%
^{\infty}jQ_{j}z^{j}$ given by
\[
z_{s}^{-1}=\lim_{j\rightarrow\infty}(Q_{j})^{1/j}.
\]
Define the function $F$%
\[
F(z)=\frac{\sum_{j=0}^{\infty}jQ_{j}z^{j}}{\sum_{j=0}^{\infty}Q_{j}z^{j}}.
\]

\begin{proposition}
The function $F(z)$ is strictly increasing on $0<z<z_{s}$.
\end{proposition}

\begin{proof}
Since $z<z_{s}$ the series $\sum_{j=0}^{\infty}Q_{j}z^{j}$ and $\sum
_{j=0}^{\infty}jQ_{j}z^{j}$ can be differentiated term by term.%
\[
\frac{dF(z)}{dz}=\frac{\sum_{j=0}^{\infty}j^{2}Q_{j}z^{j-1}\sum_{k=0}^{\infty
}Q_{k}z^{k}-\sum_{j=0}^{\infty}jQ_{j}z^{j}\sum_{k=0}^{\infty}kQ_{k}z^{k-1}%
}{\left(  \sum_{j=0}^{\infty}Q_{j}z^{j}\right)  ^{2}}.
\]
Using the symmetry of the sum in the first term of the numerator one has%
\[
\frac{dF(z)}{dz}=\frac{\sum_{j=0}^{\infty}\sum_{k=0}^{\infty}Q_{j}%
Q_{k}z^{j+k-1}((j^{2}+k^{2})/2-jk)}{\left(  \sum_{j=0}^{\infty}Q_{j}%
z^{j}\right)  ^{2}}.
\]
Since $(j^{2}+k^{2})/2\geq jk$ holds for any $j,k\geq0$ the numerator is
positive and hence $\frac{dF(z)}{dz}>0,$ proving the Proposition$.$
\end{proof}

Now, we define the critical mass density $\rho_{s}$ as
\[
\rho_{s}=\eta\sup_{z\leq z_{s}}F(z)
\]
Then, for a given $\rho<\rho_{s}$ there is a unique value of $z(\rho,\eta)$
satisfying the equality $F(z(\rho,\eta))=\frac{\rho}{\eta}$. This in turn
uniquely determines $c_{0}^{e}$ by $c_{0}^{e}=\frac{\eta}{\sum_{j=0}^{\infty
}Q_{k}z(\rho,\eta)^{j}}$ as well as $A=\sum_{j=0}^{\infty}a_{j}c_{j}^{e}$ and
$B=\sum_{j=1}^{\infty}b_{j}c_{j}^{e}$ in terms of $z(\rho,\eta).$ Hence, we
have proved

\begin{proposition}
Let $\rho,\eta<\infty$ be given$.$ Then, if $\rho<\rho_{s}$ the EDG system
admits a unique equilibrium distribution $c^{e}(\rho)$ given by%
\[
c_{j}^{e}=\frac{Q_{j}z(\rho,\eta)^{j}\eta}{\sum_{k=0}^{\infty}Q_{j}z(\rho
,\eta)^{j}}.
\]
If $\rho>\rho_{s},$ then there is no equilibrium state with density $\rho.$
\end{proposition}

Next, we define some functions which will be useful in the analysis. Consider
the function $G(c)=\sum_{j=0}^{\infty}c_{j}(\ln(c_{j})-1)$ which has the form
of entropy. It can be shown easily that it is weak$\ast$ continuous on $X_{1}%
$. Moreover, restricted to the ball $B_{\rho}=\{x\in X:\left\Vert x\right\Vert
<\rho\},$ it can be shown that $G(c)$ is bounded (see \cite{Ball}). We define
the relative entropy by
\[
V(c)=G(c)-\sum_{j=1}^{\infty}jc_{j}\ln(Q_{j})^{1/j}=\sum_{j=0}^{\infty}%
c_{j}(\ln(\frac{c_{j}}{Q_{j}})-1).
\]
It is assumed throughout the paper that $z_{s}>0$ which is equivalent to
$\lim_{j\rightarrow\infty}(Q_{j})^{1/j}<\infty.$ If we further assume
$\lim\inf(Q_{j})^{1/j}>0$ then $V(c)$ becomes bounded from above and below.
Next, we define the \textit{modified} relative entropy%
\[
V_{z,y}(c)=V(c)-\ln z\sum_{j=1}^{\infty}jc_{j}-\ln y\sum_{j=0}^{\infty}c_{j}%
\]
and the set
\[
X_{1}^{+,\rho,\eta}=\{x\in X^{+}:\sum_{j=1}^{\infty}jx_{j}=\rho,\sum
_{j=1}^{\infty}x_{j}=\eta\}.
\]
The next theorem shows the relationship between the equilibrium solutions and
the minimizers of the entropy functionals and the related sets$.$

\begin{theorem}
Assume that $z_{s}<\infty$ and $\rho<\infty.$ Then,

$(i)$ If $0\leq\rho\leq\rho_{s},$ then $c(\rho,\eta)$ is the unique minimizer
of $V_{z(\rho,\eta),y(\rho,\eta)}$ on $X_{1}^{+}$ and of $V(c)$ in
$X_{1}^{+,\rho,\eta}.$ Equivalently, every minimizing sequence $c^{i}$ of $V$
on $X_{1}^{\rho,\eta}$ converges strongly in $X_{1}.$

$(ii)$ If $\rho_{s}<\rho<\infty\,,$ then the minimizing sequence $c^{i}$
converges weakly to $c(\rho_{s},\eta)$ but not strongly and
\[
\inf_{c\in X_{1}}V_{z_{s},y_{s}}(c)=V_{z_{s},y_{s}}(c_{s}).
\]

\end{theorem}

\begin{proof}
One can easily check that the function $c_{j}\rightarrow c_{j}\left(
\ln(\frac{c_{j}}{Q_{j}z(\rho,\eta)^{j}y(\rho,\eta)})-1\right)  $ has the
minimum at $c_{j}=Q_{j}z(\rho,\eta)^{j}y(\rho,\eta)$ and hence the function
$V_{z(\rho,\eta),y(\rho,\eta)}(c)$ is minimized (over $X_{1})$ exactly at the
equilibrium distribution $c_{j}^{e}(\rho,\eta)$. Clearly, $c_{j}^{e}(\rho
,\eta)$ is also a minimizer of $V(c)$ on the set $X_{1}^{\rho,\eta}=\{c\in
X:\sum_{j=1}^{\infty}jc_{j}=\rho,$ $\sum_{j=1}^{\infty}c_{j}=\eta\}$.

Now, because $c_{j}^{i}$ is bounded on $X_{1}^{\rho,\eta}$ and because
$c_{j}^{i}\rightarrow c_{j}^{e},$ one has the weak convergence of the
minimizing sequence to the equilibrium solution. On the set $X_{1}^{\rho,\eta
},$ mass of the sequence is constant and thanks to Lemma \ref{mag2norm}, one
gets $\left\Vert c_{j}^{i}\right\Vert \rightarrow\left\Vert c_{j}%
^{e}\right\Vert $.

For the second part of the theorem, let $\rho>\rho_{s}$ and consider a special
sequence $c^{i}\in X_{1}^{\rho,\eta}$ (as in \cite{Ball}) defined by
\[
c_{j}^{i}=c_{j}^{e}(\rho_{s},\eta)+\delta_{ij}(\frac{\rho-\rho_{s}}{i})
\]

It is clear that $c^{i}\rightarrow^{\ast}c^{e}(\rho_{s},\eta).$ Also, it can
be shown by straightforward computation that $V_{z_{s,y_{s}}}(c^{i}%
)\rightarrow V_{z_{s},y_{s}}(c^{e}(\rho_{s},\eta)).$ However the convergence
cannot be strong as $\left\Vert c^{i}\right\Vert =\rho>\rho_{s}=c^{e}(\rho
_{s},\eta).$
\end{proof}

In the sequel, it will be important to know the continuity property of
$V_{z,s}(c).$ We have the following.

\begin{proposition}
\label{V-cont}$V_{z,y}(c)$ is weak$\ast$ continuous in $X_{1}$ if
$\lim_{j\rightarrow\infty}(Q_{j})^{1/j}$ exists and $z=z_{s}.$
\end{proposition}

\begin{proof}
Recall that a function $W(c)=\sum_{j=1}^{\infty}g_{j}c_{j}$ is weak$\ast$
continuous if and only if $g_{j}=o(j).$ Then, for $V_{z,y}(c)=V(c)-\ln
z\sum_{j=1}^{\infty}jc_{j}-\ln y\sum_{j=0}^{\infty}c_{j}$ to be weak$\ast$
continuous we need $\ln Q_{j}+j\ln(z)+y=o(j)$ which amounts to showing that
$\lim_{j\rightarrow\infty}\frac{\ln(Q_{j}z^{j}y)}{j}=0.$ But this follows if
$\lim_{j\rightarrow\infty}(Q_{j})^{1/j}z=1$, that is, $z=z_{s}$.
\end{proof}

\textbf{Remark:} Recall from the earlier discussions that $z_{s}$ is the
radius of convergence of the series $\sum_{k=0}^{\infty}Q_{k}z^{k}.$ A more
direct way to compute the radius of convergence is the ratio test which gives
$z_{s}=\lim_{k\rightarrow\infty}(Q_{k}/Q_{k+1})=\lim\frac{b_{k+1}}{a_{k}}.$
So, the behavior of the equilibria (and the conditions for the dynamic phase
transition as shown in the next section) is decided by the competition in the
tendency of exchange favoring "export" against "import" of monomers
($K(j,k)=b_{j}a_{k}).$ This leads to following scenarios

$(i)$ $\lim\frac{b_{k+1}}{a_{k}}=\infty,$ (\textit{exporting wins over
importing or the system favors smaller clusters}): In this case $z_{s}%
=\infty.$ Hence, for any initial mass the system can support equilibrium.

$(ii)$ $\lim\frac{b_{k+1}}{a_{k}}=\alpha>0$ (\textit{exporting and importing
are comparable}): In this case $z_{s}=\alpha$ and whether the system can
always support an equilibrium depends on the whether $\sum_{k=0}^{\infty}%
Q_{k}z_{s}^{k}=\rho_{s}$ is finite. If $\rho>\rho_{s}$ then there will be no equilibrium.

$(iii)$ $\lim\frac{b_{k+1}}{a_{k}}=0$ (\textit{importing wins over
exporting}): In this case $z_{s}=0$ and hence there is no equilibrium
irrespective of initial mass.

\subsection{Lyapunov Functions and Asymptotic Behavior}

In this section we show the convergence of solutions, under suitable
conditions, to equilibrium in the strong and weak$\ast$ senses. The approach
is similar to \cite{Ball}. The main object of use will be the
relative\ entropy $V(c)$ whose minimization was discussed in the previous
section. Our hope is that evolving in time $c(t_{i})$ becomes the minimizing
sequence for $V.$ It is therefore important to know how $V$ will behave in
time. We first state an elementary result whose proof follows easily from the
points made after Definition 4.

\begin{lemma}
\label{G-cont}The function $G(c)=\sum_{j=0}^{\infty}c_{j}(\ln(c_{j})-1)$ is
finite and weak$\ast$ continuous in $X_{1}.$
\end{lemma}

We also quote a preliminary result from \cite{EE} that guarantees the
positivity of the cluster densities.

\begin{proposition}
\label{c-pos} Let $c^{N}$ solve the truncated EDG system (\ref{Tode0}%
)-(\ref{TodeIC}) and $c_{j}^{N}(0)>0$ for some $j.$ Then $c_{j}^{N}(t)>0$ for
any $t>0.$
\end{proposition}

Note that the same result holds for the solution $c(t)$ of the original
infinite system (\ref{0-infode})-(\ref{infIC}). Next we need need the
following lemma which will be needed to show that the relative entropy is non-increasing.

\begin{lemma}
\label{D-sign} Let $a_{j},b_{j},c_{j}$ be a sequence of non-negative numbers
with $j\geq0$. Let, for a given integer $N\geq0,$ $A^{N+1}=\sum_{j=0}^{N}%
a_{j}c_{j}$ and $B^{N+1}=\sum_{j=1}^{N+1}b_{j}c_{j}.$ Define $I_{j}%
^{N+1}=a_{j}c_{j}B^{N+1}-b_{j+1}c_{j+1}A^{N+1}$ for $0\leq j\leq N$ and zero
otherwise. Then one has the inequality%
\[
D^{N+1}(c):=-\sum_{j=0}^{N+1}(I_{j-1}^{N+1}-I_{j}^{N+1})\ln(\frac{c_{j}}%
{Q_{j}})\geq0.
\]

\end{lemma}

\begin{proof}
We prove this by recursively summing the terms. Let $I_{j-1}^{N+1}-I_{j}%
^{N+1}=R_{j}^{N+1}$. From the definitions, we can relate $R_{j}^{N}$ and
$R_{j}^{N+1}.$ For the "lower boundary" term ($j=0)$,%
\begin{equation}
R_{0}^{N+1}=R_{0}^{N}+0-(a_{0}c_{0}b_{N+1}c_{N+1}-b_{1}c_{1}a_{N}c_{N}%
)\ln(\frac{c_{0}}{Q_{0}})\text{, \ \ }j=0. \label{R-0}%
\end{equation}
The middle terms are related by%
\begin{equation}
R_{j}^{N+1}=R_{j}^{N}+\left[  (a_{j-1}c_{j-1}b_{N+1}c_{N+1}-b_{j}c_{j}%
a_{N}c_{N})-(a_{j}c_{j}b_{N+1}c_{N+1}-b_{j+1}c_{j+1}a_{N}c_{N})\right]
\ln(\frac{c_{j}}{Q_{j}}). \label{R-j}%
\end{equation}
The "upper boundary" $j=N,N+1$ are then related by%
\begin{equation}
R_{N}^{N+1}=R_{N}^{N}+\left[  (a_{N-1}c_{N-1}b_{N+1}c_{N+1}-b_{N}c_{N}%
a_{N}c_{N})-(a_{N}c_{N}B^{N+1}-b_{N+1}c_{N+1}A^{N+1})\right]  \ln(\frac{c_{N}%
}{Q_{N}}), \label{R-N}%
\end{equation}%
\begin{equation}
R_{N+1}^{N+1}=(a_{N}c_{N}B^{N+1}-b_{N+1}c_{N+1}A^{N+1})\ln(\frac{c_{N+1}%
}{Q_{N+1}}). \label{R-Nplus}%
\end{equation}
Now, for adjacent indices $j,j+1$ we combine the second term (in bracket) of
$j^{th}$ equation with the first term ($j+1)^{th}$ equation which gives%
\begin{equation}
(a_{j}c_{j}b_{N+1}c_{N+1}-b_{j+1}c_{j+1}a_{N}c_{N})\ln(\frac{c_{j+1}}{Q_{j+1}%
}\frac{Q_{j}}{c_{j}}). \label{j-jpl-comb}%
\end{equation}

Next, we expand the $A^{N+1},B^{N+1}$ terms in equations (\ref{R-N}),
(\ref{R-Nplus}) noting that $a_{N}c_{N}B^{N+1}-b_{N+1}c_{N+1}A^{N+1}%
=a_{N}c_{N}B^{N}-b_{N+1}c_{N+1}A^{N}.\,$\ Combining the ($j+1)^{th}$ and terms
in (\ref{R-N}), (\ref{R-Nplus}) (inside the bracket) we get
\begin{equation}
(a_{N}c_{N}b_{j+1}c_{j+1}-b_{N+1}c_{N+1}a_{j}c_{j})\ln(\frac{c_{N+1}}{Q_{N+1}%
}\frac{Q_{N}}{c_{N}}). \label{N-Np-comb}%
\end{equation}

Now, summing over the index $j$ and recalling $\frac{Q_{j+1}}{Q_{j}}%
=\frac{a_{j}}{b_{J+1}}$ the desired sum in the statement of the lemma can be
written as
\begin{align*}
\sum_{j=0}^{N+1}R_{j}^{N+1}\ln(\frac{c_{j}}{Q_{j}})  &  =\sum_{j=0}^{N}%
R_{j}^{N}\ln(\frac{c_{j}}{Q_{j}})-\sum_{j=0}^{N+1}(a_{j}c_{j}b_{N+1}%
c_{N+1}-b_{j+1}c_{j+1}a_{N}c_{N})\ln(\frac{c_{N+1}b_{N+1}a_{j}c_{j}}%
{c_{N}c_{N}b_{j+1}c_{j+1}})\\
&  \leq\sum_{j=0}^{N}R_{j}^{N}\ln(\frac{c_{j}}{Q_{j}}).
\end{align*}
The second line followed since $(x-y)\ln(\frac{x}{y})>0$ for any real number
pairs $x,y\geq0.$ Repeating the arguments for $j\leq N$ and reducing the index
number we find
\[
\sum_{j=0}^{N+1}R_{j}^{N+1}\ln(\frac{c_{j}}{Q_{j}})\leq(a_{0}c_{0}-b_{1}%
c_{1})\ln(\frac{c_{1}}{c_{0}}\frac{b_{1}}{a_{0}})\leq0.
\]
which completes the proof.
\end{proof}

\begin{theorem}
\label{Lyap}Let $a_{j},b_{j}=O(j/\ln j)$ and $c_{j}(t)$ be the solution of
(\ref{0-infode})-(\ref{infIC}). Assume that $c_{j}(0)>0$ for some $j$ and
$0<\lim_{j\rightarrow\infty}(Q_{j})^{1/j}<\infty$ holds$.$ Then%
\begin{equation}
V(c(t))=V(c(0))-\int_{0}^{t}D(c(s))ds \label{V-integ}%
\end{equation}
where $D(c)\geq0$ and is given by%
\begin{equation}
D(c):=\sum_{0}^{\infty}(Ba_{j}c_{j}-Ab_{j+1}c_{j+1})\ln(\frac{a_{j}c_{j}%
}{b_{j+1}c_{j+1}}) \label{D(c)}%
\end{equation}

\end{theorem}

\begin{proof}
Consider the truncation of (\ref{0-infode})-(\ref{infIC}) and define
\[
V^{N}(c)=\sum_{j=0}^{N}c_{j}(\ln(\frac{c_{j}}{Q_{j}})-1).
\]
Differentiating this we get
\begin{equation}
\dot{V}^{N}(c)=\sum_{j=0}^{N}\dot{c}_{j}\ln(\frac{c_{j}}{Q_{j}})=D^{N}%
(c)+I_{N}(c)\ln(\frac{c_{N}}{Q_{N}}). \label{Vn-dot}%
\end{equation}
where $I_{N}(c)=a_{N}c_{N}B(c)-b_{N+1}c_{N+1}A(c)$. Also, by the assumption of
the theorem $a_{j}\leq C\frac{j}{\ln(j)}.$ Now, integrating both sides of
(\ref{Vn-dot}) we have%
\begin{equation}
V^{N}(c(t))=V^{N}(c(0))-\int\sum_{j=0}^{N-1}(a_{j}c_{j}B-b_{j+1}c_{j+1}%
A)\ln(\frac{c_{j+1}b_{j+1}}{a_{j}c_{j}})+\int I_{N}(c)\ln(\frac{c_{N}}{Q_{N}%
}). \label{Vn-int}%
\end{equation}
We need to show that the right hand side of (\ref{Vn-int}) converges to
(\ref{V-integ}). It is obvious that $A,B$ are bounded. Also, for $j$ large
enough we have, \thinspace$c_{j}\leq C/j^{2}$ giving $\left\vert \ln
c_{j}\right\vert \leq C\ln(j)$. These imply, for the second integrand in
(\ref{Vn-int}), we have
\[
\left\vert \left(  a_{N}c_{N}B(c)-b_{N+1}c_{N+1}A(c)\right)  \ln(\frac{c_{N}%
}{Q_{N}})\right\vert \leq C\frac{Nc_{N}}{\ln N}\left\vert \ln(\frac{c_{N}%
}{Q_{N}})\right\vert \leq CNc_{N}.
\]
Since $Nc_{N}(t)$ converges uniformly to zero on finite intervals the second
integral in (\ref{Vn-int}) vanishes in the limit $N\rightarrow\infty$. To
prove the claim of the theorem we need to show that the remaining integrand on
the right hand side of (\ref{Vn-int}) converges to (\ref{D(c)}). It is
sufficient to show that the term $\sum_{j=n}^{\infty}a_{j}c_{j}\ln(c_{j}%
a_{j})$ goes uniformly to zero (the other terms in can be done similarly).
\begin{align}
\left\vert \sum_{j=n}^{\infty}a_{j}c_{j}\ln(c_{j}a_{j})\right\vert  &
\leq\sum_{j=n}^{\infty}a_{j}c_{j}\left\vert \ln(c_{j})\right\vert +\sum
_{j=n}^{\infty}a_{j}c_{j}\left\vert \ln(a_{j})\right\vert \\
&  \leq\sum_{j=n}^{\infty}C\frac{j}{\ln(j)}c_{j}\ln(j)+\sum_{j=n}^{\infty
}C\frac{j}{\ln(j)}c_{j}\left\vert \ln(C\frac{j}{\ln(j)})\right\vert .
\end{align}

Clearly, the right hand side in the second line uniformly goes to zero and
hence (\ref{Vn-int}) does converge to (\ref{V-integ}). Finally, the
non-negativity of $D(c)$ follows as we notice that the right hand side of
(\ref{Vn-dot}) has exactly the same structure as in Lemma (\ref{D-sign}) which
is non-negative. This completes the proof.
\end{proof}

For the integral equality (\ref{V-integ}) the bounds on the export and import
rates $a_{j},b_{j}=O(j/\ln j)$ were essential while they are not needed for
the well posedness as discussed in Section 2. It would be nice, therefore, to
have a similar result for the more general case when $a_{j},b_{j}=O(j)$ only.
The following corollary provides that.

\begin{corollary}
\label{Lyap-int}Let $a_{j},b_{j}=O(j)$. Assume that $c_{j}(0)>0$ and
$0<\lim_{j\rightarrow\infty}(Q_{j})^{1/j}<\infty$ holds$.$ Let $c_{j}(t)$ be
the solution of (\ref{0-infode})-(\ref{infIC}). Then%
\begin{equation}
V(c(t))\leq V(c(0))-\int_{0}^{t}D(c(s))ds.
\end{equation}

\end{corollary}

\begin{proof}
Take the truncated system (\ref{Tode0})-(\ref{TodeIC}) and the approximation
$V^{N}$%
\begin{equation}
V^{N}(c^{N}(t))=V(c^{N}(0))-\int\sum_{j=0}^{N-1}(a_{j}c_{j}^{N}B^{N}%
(c^{N})-b_{j+1}c_{j+1}^{N}A^{N}(c^{N}))\ln(\frac{c_{j+1}^{N}b_{j+1}}%
{a_{j}c_{j}^{N}}).
\end{equation}
Fix $n\in\mathbb{N}$ and consider the subsequence $N(k)>n~$which converges to
the solution of the original EDG system. By Lemma \ref{D-sign} $D^{N(k)-1}%
(c^{N(k)})\geq D^{n}(c^{N(k)}).$ Then, since $\lim_{N(k)\rightarrow\infty
}V(c^{N(k)}(0))=V(c(0))$ one has
\[
\lim\inf D^{N(k)-1}(c^{N(k)})\geq\lim_{N(k)\rightarrow\infty}D^{n}%
(c^{N(k)})=D^{n}(c).
\]
Also, by the condition $0<\lim_{j\rightarrow\infty}(Q_{j})^{1/j}<\infty$ and
the strong convergence of $c^{N(k)}(t)$ to $c(t)$ (mass conservation and Lemma
\ref{mag2norm}) and Proposition \ref{G-cont}\ we have%
\begin{align*}
V(c(t))  &  =\lim_{N(k)\rightarrow\infty}V(c^{N(k)}(t))=\lim_{N(k)\rightarrow
\infty}V(c^{N(k)}(0))-\lim\inf\int D^{N(k)-1}(c^{N(k)}(s))ds\\
&  \leq V(c(0))-\int_{0}^{t}D^{n}(c(s))ds.
\end{align*}
Passing to the limit $n\rightarrow\infty$ yields the result.
\end{proof}

For the asymptotic behavior we will study the positive orbit of the flow
$O^{+}(\phi)=\cup_{t\geq0}\phi(t)$ where $\phi(t)=T(t)c(0).$ We define the
$\omega-$limit set by $\omega(\phi)=\{x\in X:\phi(t_{j})\rightarrow x$ for
some sequence $t_{j}\}$. We quote the following result from the general theory
which is standard.

\begin{proposition}
Suppose that $O^{+}(\phi)$ is relatively compact. Then $\omega(\phi)$ is non
empty, quasi-invariant and $\lim_{t\rightarrow\infty}dist(\phi(t),\omega
(\phi))=0.$
\end{proposition}

We can now prove the main theorems of this section. The first theorem below
shows the weak$\ast$ convergence under fairly general conditions.

\begin{theorem}
\label{weak-conv}Consider the system (\ref{infode})-(\ref{infIC}) with
$K(j,k)=b(j)a(k)$ and $V(c(0))<\infty.$ Let $a_{j},b_{j}=O(j/\ln j)$ for large
$j.$ Let the initial density be given $\rho_{0}=\sum_{k=1}^{\infty}%
kc_{k}(0)<\infty$ and assume also that $\lim_{j\rightarrow\infty}\frac
{b_{j+1}}{a_{j}}=z_{s}$ $(0<z_{s}<\infty).$ Then $c(t)\rightharpoonup^{\ast
}c^{\rho}$ for some $\rho$ with $0\leq\rho\leq\min(\rho_{0},\rho_{s}).$
\end{theorem}

\begin{proof}
Consider the function $V_{z_{s},y_{s}}(c).$ From Proposition \ref{V-cont} it
is continuous on $B_{\rho_{0}}.$ Also since total mass density $\sum
_{k=1}^{\infty}kc_{k}(t)$ and total number densities $\sum_{k=1}^{\infty}%
c_{k}(0)$ are conserved by Theorem \ref{Lyap} we have%
\[
V_{z_{s},y_{s}}(c(t))=V_{z_{s},y_{s}}(c(0))-\int_{0}^{t}D(c(s))ds
\]
Boundedness of $\sum_{k=1}^{\infty}kc_{k}(t)$ also implies that $O^{+}(c)$ is
relatively compact in $B_{\rho_{0}}$. By the invariance principle $\omega(c)$
is non empty and consists of points $V_{z_{s},y_{s}}(c)=const$ which implies
that, for any element in $\bar{c}\in\omega(c),$ $D(\bar{c})=0$ and hence
$\bar{c}$ has the form $\bar{c}_{r}=Q_{j}\left(  \frac{B(\bar{c})}{A(\bar{c}%
)}\right)  ^{j}\bar{c}_{0}(t)$ for some $\bar{c}(0)\in\omega(c).$ But, this is
exactly the form of equilibrium solutions. Since the mass density cannot
increase it follows that $\omega(c)$ consists of equilibria $c^{\rho,\eta}$
with $0<\rho\leq\rho_{s}.$ By the previous proposition $dist(c(t),c^{\rho
})\rightarrow0$ as $t\rightarrow0,$ completing the proof$.$
\end{proof}

We can strengthen the theorem for the subcritical case by making further
assumptions on the strength of "export" tendency over the "import" in the
system. More precisely, let%
\[
(H1)\text{ }\lim_{j\rightarrow\infty}\frac{a_{j}}{b_{j+1}}=0
\]
hold. Then we can prove the following strong convergence result.

\begin{theorem}
Let $c_{j}(t)$ solve the system (\ref{0-infode})-(\ref{infIC}). Assume that
$a_{j},b_{j}=O(j/\ln j)$ and H1 holds$.$ Then $c(t)\rightarrow c^{\rho}$
strongly in $X_{1}.$
\end{theorem}

\begin{proof}
$H1$ implies that the radius of convergence of the series $z_{s}=\infty$ which
is equivalent to $\lim_{j\rightarrow\infty}(Q_{j})^{1/j}=0.$ By the
monotonicity of $V(c)$ one has $V(c(t))\leq V(c(0))$. Also, by Proposition
\ref{G-cont} $\sum_{j=0}^{\infty}c_{j}(t)\ln(c_{j})<\infty.$ Hence we have
\[
-\sum_{j=0}^{\infty}jc_{j}(t)\ln(Q_{j})^{1/j}\leq C.
\]

Since $-\ln\left(  (Q_{j})^{1/j}\right)  \rightarrow\infty$ by $H1,$ it
follows that $O^{+}(c)$ is relatively compact in $X_{1}$ and the by the
invariance principle the $c(t)$ converges strongly to a distribution in
$\omega(c)$ in the form $\bar{c}_{j}(t)=Q_{j}\left(  \frac{B(\bar{c})}%
{A(\bar{c})}\right)  ^{j}\bar{c}_{0}(t)$ where $\bar{c}_{j}(0)=\lim
_{j\rightarrow\infty}c_{j}(t)$ has the form of equilibrium solutions$.$ By the
conservation of number and mass density in time, i.e., $\sum_{j=0}^{\infty
}\bar{c}_{j}(t)=\eta,$ $\sum_{j=1}^{\infty}j\bar{c}_{j}(t)=\rho_{0}$ and the
uniqueness of equilibrium solutions one concludes that $\omega(c)$ consists of
single point, that is, the equilibrium solutions that correspond to the pair
$(\rho,\eta).$
\end{proof}

If the \textit{exporting} and \textit{importing} tendencies are comparable as
in Remark 1 Case $(ii)$, then the above argument does not work and we need
extra conditions to secure the strong convergence. We will need to control the
moments of the initial distribution and crucially make use of a uniform
comparison of $b_{j},a_{j}$ which will replace $(H1)$ i.e.,%
\[
(H2)\text{ \ }\frac{b_{j}}{a_{j}}\geq z_{s}\text{ \ for }j\geq1.
\]

\begin{theorem}
\label{qual-conv2-str}Let $c_{j}(t)$ solve the system (\ref{0-infode}%
)-(\ref{infIC}) and $\rho_{0}\leq\rho_{s}.$ Let $b_{j}\geq Cj^{\lambda}$
$(1>\lambda>0)$ and $\sum_{j=0}^{\infty}j^{p}c_{j}(0)<\infty$ for some
$p>2-\lambda.$ Assume further that H2 holds and $a_{j},b_{j}=O(j/\ln j).$ Then
$c(t)\rightarrow c^{\rho}$ strongly in $X_{1}.$
\end{theorem}

\begin{proof}
The main line of argument, as in the previous theorem, is to show that
$O^{+}(c)$ is relatively compact in $X_{1}.$ This will follow by showing that
$M_{m}(t)<C$ for some $m>1.$ Consider the $p^{th}$ moment of the system
$M_{p}:=\sum_{j=0}^{\infty}j^{p}c_{j}(t)$ with ($1<p<2).$ By Theorem 1,
$M_{p}(t)<\infty$ for any $t<\infty$. Now, choose $m<p$ such that
$m>2-\lambda$ still holds. By Lemma 1, one has%
\[
\dot{M}_{m}=\sum_{j\geq1}((j-1)^{m}-j^{m})b_{j}c_{j}A+\sum_{j\geq0}%
((j+1)^{m}-j^{m})a_{j}c_{j}B.
\]
Taylor expanding the $(j-1)^{m}$ and $(j+1)^{m}$ terms up to second order we
find
\[
\dot{M}_{m}\leq-\sum_{j\geq1}mj^{m-1}b_{j}c_{j}A+\sum_{j\geq1}mj^{m-1}%
a_{j}c_{j}B+CAB.
\]
Note that $A,B$ depend on time. By weak$\ast$ convergence (Theorem
\ref{weak-conv}) $c_{j}(t)\rightarrow c_{j}^{e}$ as $t\rightarrow\infty.$
Then, since $\sum_{j\geq1}jc_{j}\leq C,$ one has $\lim_{t\rightarrow\infty
}\sum_{j\geq1}g_{j}c_{j}(t)\rightarrow\sum_{j\geq1}g_{j}c_{j}^{e}$ for any
$g_{j}=o(j).$ Therefore, it follows, since $a_{j},b_{j}=o(j)$, that
$\frac{B(c(t))}{A(c(t))}\rightarrow\frac{B(c^{e})}{A(c^{e})}=z(\rho)<z_{s}$.

Now, since $\frac{b_{j}}{a_{j}}\geq z_{s}$ by the assumption in the theorem,
there is a $t_{\ast}$ and $\delta>0$ such that $-\frac{A}{B}+\frac{a_{j}%
}{b_{j}}\leq-\delta$ and
\[
\dot{M}_{m}\leq C+m\sum_{j\geq1}(-b_{j}c_{j}A+a_{j}c_{j}B)j^{m-1}\leq
C+m\sum_{j\geq1}(-\frac{A}{B}+\frac{a_{j}}{b_{j}})j^{m-1}Bb_{j}c_{j}.
\]
By the fact that $B(c)>\varepsilon$ for some $\varepsilon>0$ (for $t_{\ast}$
large enough) and the condition $b_{j}\geq Cj^{\lambda}$ $(0<\lambda<1)$ we
find%
\[
\dot{M}_{m}\leq C-C\delta\sum_{j\geq1}j^{m-1+\lambda}c_{j}.
\]
Integrating both sides and noting $M_{m-1+\lambda}\leq M_{m}$ we get%
\[
M_{m-1+\lambda}(t)\leq C(t-t_{\ast})+M_{m}(t_{\ast})-C\int_{t_{\ast}}%
^{t}M_{m-1+\lambda}(s)ds.
\]
Comparing this to the solution of $x(t)=C(t-t_{\ast})+M_{m}(t_{\ast}%
)-C\int_{t_{\ast}}^{t}x(s)ds$ we find $M_{m-1+\lambda}(t)\leq C$ for all
$t>t_{\ast}.$ Since $\lambda>0$ and $m>2-\lambda$ by our choice, it follows
that the tail of the distribution $jc_{j}$ uniformly approaches to zero giving
the compactness of the orbit in $X_{1}$ and hence showing $c(t)\rightarrow
c^{\rho}$ strongly.
\end{proof}

\textbf{Remark:} Without essentially changing the proof, the hypothesis $(H2)$
could be replaced with $\frac{b_{j}}{a_{j}}\geq z_{s}$ for finitely many $j$ values.

\section{CONVERGENCE TO EQUILIBRIUM WITHOUT DETAILED BALANCE}

In this section we extend the study of convergence of time dependent solutions
to equilibrium without imposing a structure condition on the equilibria. Our
goal is to obtain explicit convergence rates to equilibrium. We assume,
throughout this section
\[
(H3)\text{ \ \ }K(j,k)=ja_{k}+b_{j}+\varepsilon\beta_{j}\alpha_{k}\text{ with
}a_{j}\geq\tilde{a}>0\text{.}%
\]
Depending on the type of assumptions we obtain two different convergence
results. Each result relies on a contraction property of the time dependent
solution. The first contraction property is a consequence of the monotonicity
of the $a_{j},b_{j}$ functions which leads to exponentially fast convergence
in the "weak" metric ($dist(c,d)=\left\Vert c-d\right\Vert _{0}=\sum_{j\geq
0}\left\vert c_{j}-d_{j}\right\vert ).$ The second contraction property
follows from the total mass of the system being sufficiently small and is used
to show exponentially fast convergence in the "strong" metric ($\left\Vert
c-d\right\Vert _{1}=\sum_{j\geq1}j\left\vert c_{j}-d_{j}\right\vert )$. Such a
contraction property was first shown to hold for the coagulation-fragmentation
systems under a similar small mass assumption \cite{Fournier}.

\subsection{Exponentially Fast Weak Convergence to Equilibrium}

In this subsection our approach is partly motivated by that, in the EDG
equations, $a_{j}$ represents the import rates of particle (and hence causes
growth of clusters) and $b_{j}$ represents the export rate (and hence causes
breakdown of clusters). If such an interpretation was meaningful then one
would expect that for monotonically increasing $b_{j}$ (in $j)$ and
monotonically decreasing $a_{j}$ the dynamics favor the approach to
equilibrium which would be manifested in the convergence rates.

\begin{theorem}
\label{expo-conv}Consider the EDG system (\ref{0-infode})-(\ref{infIC}). Let
the hypothesis of Theorem 1 be satisfied with a given initial mass $\rho.$ Let
the kernel has the form H3 with $a_{j}$ is non-increasing, $b_{j}$
non-decreasing, $\alpha_{j},\beta_{j}$ are bounded with $\varepsilon>0$ is
small. Then the solutions of (\ref{0-infode})-(\ref{infIC}) converge to a
unique equilibrium in the sense that
\begin{equation}
\sum_{j\geq1}\left\vert c_{j}(t)-c_{j}^{e}\right\vert \leq4\rho e^{-\gamma t},
\label{tot_num-conv}%
\end{equation}
where $\gamma(\tilde{a},\varepsilon)>0$ can be computed explicitly.
\end{theorem}

The main idea of the theorem (covered in the next lemma) is based on defining
an appropriate positive time dependent quantity which measures the distance
between two solutions that have the same mass and showing that this distance
contracts in time, i.e., two solutions approach to each other. It will then be
shown that the limit solution is actually the equilibrium.

To prove the contraction, we will focus on the evolution of the tail of the
distributions defined as $C_{j}(t)=\sum_{k\geq j}c_{k}(t).$ This approach
proved useful in Becker-Doring systems \cite{Ball},\cite{Phil},\cite{Can1} and
were also recently adopted to prove some of the key properties of the EDG
system such as nonexistence \cite{EE} and extension of uniqueness results
without additional moment assumptions \cite{Sch}.

\begin{lemma}
\label{L-expo-conv}Consider two solutions $c_{j},d_{j}$ of the system
(\ref{0-infode})-(\ref{infIC}) with the same initial mass. Let the conditions
of Theorem 1 be satisfied and hypothesis H3 hold. Assume further that $a_{j}$
is non-increasing, $b_{j}$ non-decreasing, $\alpha_{j},\beta_{j}$ are bounded
and $\varepsilon>0$ is small enough$.$ Then the solutions of (\ref{0-infode}%
)-(\ref{infIC}) approach to each other exponentially fast as
\[
\sum_{j\geq1}\left\vert c_{j}(t)-d_{j}(t)\right\vert \leq4\rho e^{-\gamma t}%
\]

\end{lemma}

\begin{proof}
We first consider the dynamics for $C_{j}$, the tail of $(c_{j})_{j=1}%
^{\infty}.$ By direct computation the evolution equation for $C_{j}$ is
\begin{align*}
\dot{C}_{j}  &  =\sum_{k\geq1}K(k,j-1)c_{k}c_{j-1}-\sum_{k\geq0}%
K(j,k)c_{k}c_{j}\\
&  =\sum_{k\geq1}(ka_{j-1}+b_{k})c_{k}c_{j-1}-\sum_{k\geq0}(ja_{k}+b_{j}%
)c_{k}c_{j}\\
&  +\varepsilon\sum_{k\geq1}\alpha_{j-1}\beta_{k}c_{k}c_{j-1}-\varepsilon
\sum_{k\geq0}\beta_{j}\alpha_{k}c_{k}c_{j}.
\end{align*}
Taking the sum over $"k"$ and denoting, as before, $A(c)=\sum_{j\geq0}%
a_{j}c_{j},$ $B(c)=\sum_{j\geq1}b_{j}c_{j}$ and defining $\tilde{A}%
(c)=\sum_{j=0}^{\infty}\alpha_{j}c_{j}$, $\tilde{B}(c)=\sum_{j=1}^{\infty
}\beta_{j}c_{j}$ one gets
\[
\dot{C}_{j}=\rho a_{j-1}c_{j-1}+B(c)c_{j-1}-jc_{j}A(c)-b_{j}c_{j}%
+\varepsilon\alpha_{j-1}c_{j-1}\tilde{B}(c)-\varepsilon\beta_{j}c_{j}\tilde
{A}(c),
\]
where we used $\rho=\sum_{j\geq0}jc_{j}$ and $1=\sum_{j\geq0}c_{j}.$
Similarly, for the other solution $d_{j},$ one has%
\[
\dot{D}_{j}=\rho a_{j-1}d_{j-1}+B(d)d_{j-1}-jd_{j}A(c)-b_{j}d_{j}%
+\varepsilon\alpha_{j-1}d_{j-1}\tilde{B}(d)-\varepsilon\beta_{j}d_{j}\tilde
{A}(d).
\]
Since $\rho=\sum_{j\geq0}jc_{j}=\sum_{j\geq0}jd_{j},$ one has%
\begin{align*}
\dot{E}_{j}  &  =\rho a_{j-1}e_{j-1}+\left(  B(c)c_{j-1}-B(d)d_{j-1}\right)
-\left(  jc_{j}A(c)-jd_{j}A(d)\right)  -b_{j}e_{j}\\
&  +\varepsilon\alpha_{j-1}(\tilde{B}(c)c_{j-1}-\tilde{B}(d)d_{j-1}%
)-\varepsilon\beta_{j}(c_{j}\tilde{A}(d)-c_{j}\tilde{A}(d)).
\end{align*}
Then, setting $e_{j}=E_{j}-E_{j+1},$ for the difference terms in the
parenthesis we can write
\[
c_{j}A(c)-d_{j}A(d)=e_{j}A(c)+d_{j}(A(c)-A(d))=(E_{j}-E_{j+1})A(c)+d_{j}%
(A(c)-A(d)),
\]%
\[
B(c)c_{j-1}-B(d)d_{j-1}=e_{j-1}B(c)+d_{j-1}(B(c)-B(d))=(E_{j-1}-E_{j}%
)B(c)+d_{j-1}(B(c)-B(d)).
\]
Denoting $A(c)-A(d)=A(e)$ and $B(c)-B(d)=B(e)$ (similarly for $\tilde
{A},\tilde{B})$ we find that the tail of the difference of solutions evolves
according to%
\begin{align*}
\dot{E}_{j}  &  =\rho a_{j-1}(E_{j-1}-E_{j})+(E_{j-1}-E_{j})B(c)+d_{j-1}B(e)\\
&  -j(E_{j}-E_{j+1})A(c)-jd_{j}A(e)-b_{j}(E_{j}-E_{j+1})\\
&  +\varepsilon\alpha_{j-1}\left(  \tilde{B}(c)(E_{j-1}-E_{j})+d_{j}\tilde
{B}(e)\right)  -\varepsilon b_{j}\left(  \tilde{A}(c)(E_{j}-E_{j+1}%
)+d_{j}\tilde{A}(e)\right)  .
\end{align*}

We next show that the tail of the difference goes to zero. Consider the
absolute value of the tail density $\left\vert E_{j}\right\vert .$ Taking the
time derivative we get%
\[
\frac{d\left\vert E_{j}\right\vert }{dt}=sgn(E_{j})\dot{E}_{j}%
\]%
\begin{align*}
&  =sgn(E_{j})\left(  \rho a_{j-1}(E_{j-1}-E_{j})+(E_{j-1}-E_{j}%
)B(c)+d_{j-1}B(e)\right) \\
&  +sgn(E_{j})\left(  -j(E_{j}-E_{j+1})A(c)-jd_{j}A(e)-b_{j}(E_{j}%
-E_{j+1})\right) \\
&  +\varepsilon sgn(E_{j})\left(  \alpha_{j-1}\left(  \tilde{B}(c)(E_{j-1}%
-E_{j})+d_{j}\tilde{B}(e)\right)  -\beta_{j}\left(  \tilde{A}(c)(E_{j}%
-E_{j+1})+d_{j}\tilde{A}(e)\right)  \right)  .
\end{align*}
Since $sgn(E_{j})E_{j}=\left\vert E_{j}\right\vert $ and $E_{j\pm1}%
\leq\left\vert E_{j\pm1}\right\vert ,$ summing over $j$ in both sides gives%
\begin{align}
\sum_{j=1}^{\infty}\frac{d\left\vert E_{j}\right\vert }{dt}  &  \leq\sum
_{j=1}^{\infty}\left(  \rho a_{j-1}(\left\vert E_{j-1}\right\vert -\left\vert
E_{j}\right\vert )+(\left\vert E_{j-1}\right\vert -\left\vert E_{j}\right\vert
)B(c)+d_{j-1}\left\vert B(e)\right\vert \right) \label{Ej-sum-der}\\
&  +\sum_{j=1}^{\infty}\left(  j(\left\vert E_{j+1}\right\vert -\left\vert
E_{j}\right\vert )A(c)+jd_{j}\left\vert A(e)\right\vert +b_{j}(\left\vert
E_{j+1}\right\vert -\left\vert E_{j}\right\vert )\right)  \label{Ej-sum-derB}%
\\
&  +\varepsilon\alpha_{j-1}\left(  \tilde{B}(c)(\left\vert E_{j-1}\right\vert
-\left\vert E_{j}\right\vert )+d_{j-1}\tilde{B}(e)\right)  +\varepsilon
\beta_{j}\left(  \tilde{A}(c)(\left\vert E_{j+1}\right\vert -\left\vert
E_{j}\right\vert )+d_{j}\tilde{A}(e)\right)  . \label{Ej-sum-derC}%
\end{align}

Now, let $S_{1},S_{2},S_{3}$ denote the sum of the three sums on the right
hand side of (\ref{Ej-sum-der}), $S_{4},S_{5},S_{6}$ denote the three sums in
(\ref{Ej-sum-derB}) and $S_{7},S_{8},S_{9},S_{10}$ denote the four terms in
(\ref{Ej-sum-derC}). We treat each $S_{j}$ separately. For the first term, we
have%
\[
S_{1}=\rho\sum_{j=1}^{\infty}a_{j-1}(\left\vert E_{j-1}\right\vert -\left\vert
E_{j}\right\vert )\rho=\rho a_{0}\left\vert E_{0}\right\vert +\rho\sum
_{j=1}^{\infty}(a_{j}-a_{j-1})\left\vert E_{j}\right\vert .
\]
where the term $a_{0}\left\vert E_{0}\right\vert $ is zero by the conservation
of total volume, that is, $E_{0}=\sum_{j=0}^{\infty}c_{j}-\sum_{j=0}^{\infty
}d_{j}=0.$ For the second term $S_{2}$ we find%
\[
S_{2}=B(c)\sum_{j=1}^{\infty}(\left\vert E_{j-1}\right\vert -\left\vert
E_{j}\right\vert )=0.
\]
For $S_{3}$ we first observe, since $\sum_{j=1}^{\infty}d_{j-1}=1$ (total
volume),%
\[
S_{3}=\sum_{j=1}^{\infty}d_{j-1}\left\vert B(e)\right\vert =\left\vert
B(e)\right\vert ,
\]
while $\left\vert B(e)\right\vert $ can be written as%
\begin{align*}
\left\vert B(e)\right\vert  &  =\left\vert \sum_{j=1}^{\infty}b_{j}%
e_{j}\right\vert =\left\vert \sum_{j=1}^{\infty}b_{j}(E_{j}-E_{j+1}%
)\right\vert \leq b_{1}\left\vert E_{1}\right\vert +\left\vert \sum
_{j=2}^{\infty}(b_{j}-b_{j-1})E_{j}\right\vert \\
&  \leq b_{1}\left\vert E_{1}\right\vert +\sum_{j=2}^{\infty}\left\vert
b_{j}-b_{j-1}\right\vert \left\vert E_{j}\right\vert .
\end{align*}

Next we compute the terms in (\ref{Ej-sum-derB}), $S_{4},S_{5},S_{6}.$ For the
$S_{4}$ term we find%
\[
S_{4}=\sum_{j=1}^{\infty}j(\left\vert E_{j+1}\right\vert -\left\vert
E_{j}\right\vert )A(c)=-\left\vert E_{1}\right\vert -\sum_{j=2}^{\infty
}(j-1-j)\left\vert E_{j}\right\vert =-\sum_{j=1}^{\infty}\left\vert
E_{j}\right\vert .
\]
The $S_{5}$ term reads%
\[
\sum_{j=1}^{\infty}jd_{j}\left\vert A(e)\right\vert =\rho\left\vert
A(e)\right\vert
\]
and the $\left\vert A(e)\right\vert $ term can be written as%
\begin{align*}
\left\vert A(c)-A(d)\right\vert  &  =\left\vert \sum_{j=0}^{\infty}a_{j}%
e_{j}\right\vert =\left\vert \sum_{j=0}^{\infty}a_{j}(E_{j}-E_{j+1}%
)\right\vert \leq a_{0}\left\vert E_{0}\right\vert +\left\vert \sum
_{j=1}^{\infty}(a_{j}-a_{j-1})E_{j}\right\vert \\
&  \leq\sum_{j=1}^{\infty}\left\vert a_{j}-a_{j-1}\right\vert \left\vert
E_{j}\right\vert .
\end{align*}
where again we used $E_{0}=0$. Now, shifting the indices, the $S_{6}$ term can
be written as%
\[
\sum_{j=1}^{\infty}b_{j}(-\left\vert E_{j}\right\vert +\left\vert
E_{j+1}\right\vert )=-b_{1}\left\vert E_{1}\right\vert +\sum_{j=1}^{\infty
}(b_{j-1}-b_{j})\left\vert E_{j}\right\vert .
\]
Now, we notice that, by the non-increasing property of $a_{j},$ $a_{j}%
-a_{j-1}=-\left\vert a_{j}-a_{j-1}\right\vert $ and hence $S_{1}$ and $S_{5}$
are opposite of each other and cancel out. Similarly, by the non-decreasing
property of $b_{j},$ $b_{j-1}-b_{j}=-\left\vert b_{j}-b_{j-1}\right\vert $ and
therefore $S_{3}$ and $S_{6}$ also cancel each other in the sum. Then, since
$S_{2}=0$ by computation, we are left with the following%
\begin{equation}
\sum_{j=1}^{\infty}\frac{d\left\vert E_{j}\right\vert }{dt}\leq-A(c)\sum
_{j=1}^{\infty}\left\vert E_{j}\right\vert +S_{7}+S_{8}+S_{9}+S_{10}.
\label{Ej-sum-der2}%
\end{equation}
Finally, we treat the $S_{7},...,S_{10}$ terms. Setting $\beta_{0}=0$ and
repeating the manipulations done for $S_{1},...,S_{6}$ we find%
\[
S_{7}+S_{8}\leq\varepsilon\tilde{B}(c)\sum_{j=1}^{\infty}\left\vert \alpha
_{j}-\alpha_{j-1}\right\vert \left\vert E_{j}\right\vert +\varepsilon\tilde
{A}(d)\sum_{j=1}^{\infty}\left\vert \beta_{j}-\beta_{j-1}\right\vert
\left\vert E_{j}\right\vert ,
\]%
\[
S_{9}+S_{10}\leq\varepsilon\tilde{A}(c)\sum_{j=1}^{\infty}\left\vert \beta
_{j}-\beta_{j-1}\right\vert \left\vert E_{j}\right\vert +\varepsilon\tilde
{B}(d)\sum_{j=1}^{\infty}\left\vert \alpha_{j}-\alpha_{j-1}\right\vert
\left\vert E_{j}\right\vert .
\]
Now, since $\left\vert \alpha_{j}-\alpha_{j-1}\right\vert ,\left\vert
\beta_{j}-\beta_{j-1}\right\vert \leq2L$ for some $L>0,$ we have%
\[
S_{7}+S_{8}+S_{9}+S_{10}\leq8\varepsilon L^{2}\sum_{j=1}^{\infty}\left\vert
E_{j}\right\vert .
\]
Adding all terms in (\ref{Ej-sum-der2}) and using $A(c)=\sum_{j=0}^{\infty
}a_{j}c_{j}\geq\sum_{j=0}^{\infty}\tilde{a}c_{j}=\tilde{a}$ gives
\[
\sum_{j=1}^{\infty}\frac{d\left\vert E_{j}\right\vert }{dt}\leq-\tilde{a}%
\sum_{j=1}^{\infty}\left\vert E_{j}\right\vert +\varepsilon\left(  8L^{2}%
\sum_{j=1}^{\infty}\left\vert E_{j}\right\vert \right)  \leq-(\tilde{a}%
-8L^{2}\varepsilon)\sum_{j=1}^{\infty}\left\vert E_{j}\right\vert
\]
from which we deduce $\sum_{j=1}^{\infty}\left\vert E_{j}(t)\right\vert
\leq\sum_{j=1}^{\infty}\left\vert E_{j}(0)\right\vert e^{-(\tilde
{a}-\varepsilon8L^{2})t}.$ To finish the proof we observe
\[
\left\vert e_{j}\right\vert \leq\left\vert E_{j}\right\vert +\left\vert
E_{j+1}\right\vert ,
\]
and then taking the sum we arrive at
\begin{equation}
\sum_{j=0}^{\infty}\left\vert e_{j}\right\vert \leq2\sum_{j=1}^{\infty
}\left\vert E_{j}\right\vert \leq2\left(  \sum_{j=1}^{\infty}\left\vert
E_{j}(0)\right\vert \right)  e^{-(\tilde{a}-8\varepsilon L^{2})t}.
\label{ej_bd}%
\end{equation}
To finish the proof we observe that $E_{j}(0)=\sum_{k\geq j}e_{k}(0)$ and the
sum $\sum_{j=1}^{\infty}\left\vert E_{j}(0)\right\vert $ can be written as%
\begin{align*}
\sum_{j=1}^{\infty}\left\vert E_{j}(0)\right\vert  &  \leq\sum_{j=1}^{\infty
}\sum_{k\geq j}\left\vert e_{k}(0)\right\vert =\sum_{k=1}^{\infty}\sum
_{j=1}^{k}\left\vert e_{k}(0)\right\vert =\sum_{k=1}^{\infty}k\left\vert
e_{k}(0)\right\vert \\
&  \leq\sum_{k=1}^{\infty}k(c_{k}+d_{k})\leq2\rho.
\end{align*}
where in the first line we changed the order of summation. Using this in
(\ref{ej_bd}) completes the proof.
\end{proof}

As a consequence of this lemma, all solutions having the same mass will go to
the equilibrium solution as shown in the next proposition. Although the result
is obtained only for non-decreasing $b_{j},$ non-increasing $a_{j}$, the
involvement of the monotonicity gives a clear sign that the result should
generalize (see the Conclusion section).

Next, we need to ensure that the solutions are mass preserving at all times.

\begin{lemma}
\label{L-fin-mom}Under the conditions of Lemma \ref{L-expo-conv} all moments
of the EDG system \textit{(\ref{0-infode})-(\ref{infIC}) are finite.}
\end{lemma}

\begin{proof}
We make the proof for $n=2$ and the general proof is inferred by induction.
Consider the truncated system \textit{(\ref{Tode0})-(\ref{TodeIC}). }Using
Lemma 1 we have%
\begin{align*}
\dot{M}_{2}^{N}(t)  &  =\sum_{j=0}^{N-1}\sum_{k=1}^{N}((j+1)^{2}-j^{2}%
)(ka_{j}+b_{k}+\varepsilon\beta_{k}a_{j})c_{j}^{N}c_{k}^{N}\\
&  +\sum_{j=1}^{N}\sum_{k=0}^{N-1}((j-1)^{2}-j^{2})(ja_{k}+b_{j}%
+\varepsilon\beta_{j}a_{k})c_{j}^{N}c_{k}^{N}.
\end{align*}
Expanding the terms in the parenthesis we find%
\begin{align*}
\dot{M}_{2}^{N}(t)  &  \leq\sum_{j=0}^{N-1}(2j+1)a_{j}c_{j}^{N}\rho+\sum
_{j=0}^{N-1}(2j+1)c_{j}^{N}B^{N}+\varepsilon\sum_{j=0}^{N-1}(2j+1)\alpha
_{j}\tilde{B}^{N}c_{j}^{N}\\
&  +\sum_{j=1}^{N}(1-2j)jA^{N}c_{j}^{N}+\sum_{j=1}^{N}(1-2j)b_{j}c_{j}%
^{N}+\varepsilon\sum_{j=1}^{N}(1-2j)\beta_{j}\tilde{A}^{N}c_{j}^{N},
\end{align*}
where $A^{N}=\sum_{j=0}^{N-1}a_{j}c_{j}^{N},$ $B^{N}=\sum_{j=1}^{N}b_{j}%
c_{j}^{N}$ (and similarly for $\tilde{A}^{N},\tilde{B}^{N}).$ Using $a_{j}\leq
a_{0}$, $b_{j}\leq\bar{b}j$ and the bound $\alpha_{j},\beta_{j}\leq L$ we get
the inequality%
\begin{align}
\dot{M}_{2}^{N}(t)  &  \leq(2\rho+1)a_{0}\rho+L(2\rho+1)\rho+2\varepsilon
L^{2}(2\rho+1)\label{M2-der2}\\
&  +La_{0}-2\tilde{a}(\sum_{j=1}^{N}c_{k}(0))M_{2}^{N}(t)+(L-2b_{\min}%
\rho)+\varepsilon(L^{2}-2\beta_{\min}\rho)\nonumber\\
&  \leq C(\rho,a_{0},b_{\min},\beta_{\min},L,\varepsilon)-2\tilde{a}%
(\sum_{j=1}^{N}c_{k}(0))2M_{2}^{N}(t).
\end{align}

where we used the conservation of volume for the truncated system $\sum
_{j=1}^{N}c_{k}^{N}(t)=\sum_{j=1}^{N}c_{k}(0)$ and chose $N$ large enough that
$\sum_{j=1}^{N}c_{k}^{N}(0)>0.$ By Gronwall inequality we see that $M_{2}%
^{N}(t)$ is uniformly bounded$.$ Hence we can pass to the limit $N\rightarrow
\infty$ and hence $M_{2}(\infty)\leq C(\rho,a_{0},b_{\min},\beta_{\min
},L,\varepsilon)/2.$ By induction and following similar steps of computations
it can easily be shown that $M_{n}(t)$ is finite for any $n>2$.
\end{proof}

Now we can show the existence of equilibrium solutions\textit{. }

\begin{proposition}
\label{no-detB1}Let the hypothesis of hypothesis of Theorem 1 be satisfied and
(H3) hold. Assume further that $a_{j}$ is non-increasing, $b_{j}$
non-decreasing and $\alpha_{j},\beta_{j}$ are bounded and $\varepsilon>0$ is
small$.$ Then for any solution satisfying (\ref{0-infode})-(\ref{infIC})\ one
has%
\[
\lim_{t\rightarrow\infty}\left\vert \dot{c}_{j}(t)\right\vert =0\text{
\ \ }(\text{for }j\geq0).
\]

\end{proposition}

\begin{proof}
Let $d(t)=c(t+\delta).$ Then, by the contraction property (Lemma
\ref{L-expo-conv}) we have%
\begin{equation}
\sum_{j=0}^{\infty}\frac{\left\vert c_{j}(t+\delta)-c_{j}(t)\right\vert
}{\delta}\leq e^{-\gamma t}\sum_{j=0}^{\infty}\frac{\left\vert c_{j}%
(\delta)-c_{j}(0)\right\vert }{\delta}. \label{cj-expo-ineq}%
\end{equation}

To take the limit $\delta\rightarrow0$ we first observe that the term in the
right hand side can be written as%
\begin{equation}
\frac{\left\vert c_{j}(\delta)-c_{j}(0)\right\vert }{\delta}=\left\vert
\dot{c}_{j}(\delta_{j})\right\vert \label{cj-quot}%
\end{equation}
by the mean value theorem. Now, because $c_{j}\in C^{1}$ and $K(j,k)\leq
Cj(k+1)$ we have
\begin{align*}
\sum_{j=0}^{\infty}\left\vert \dot{c}_{j}(\delta_{j})\right\vert  &  \leq
\sum_{j=0}^{\infty}\sum_{k=0}^{\infty}K(j+1,k)c_{k}(\delta_{j})c_{j+1}%
(\delta_{j})+\sum_{j=0}^{\infty}\sum_{k=0}^{\infty}K(j,k)c_{k}(\delta
_{j})c_{j}(\delta_{j})\\
&  +\sum_{j=1}^{\infty}\sum_{k=1}^{\infty}K(k,j)c_{k}(\delta_{j})c_{j}%
(\delta_{j})+\sum_{j=1}^{\infty}\sum_{k=1}^{\infty}K(k,j-1)c_{k}(\delta
_{j})c_{j-1}(\delta_{j})\\
&  \leq2\rho\sum_{j=0}^{\infty}(j+1)c_{j}(\delta_{j})+2\rho\sum_{j=1}^{\infty
}jc_{j}(\delta_{j}).
\end{align*}

Because all moments are finite, as shown by the previous lemma, we have
$c_{j}\leq C/j^{3}$ in particular (uniform in time). Then, it follows that
\[
\sum_{j=0}^{\infty}\left\vert \dot{c}_{j}(\delta_{j})\right\vert \leq5\rho
\sum_{j=1}^{\infty}j\frac{C}{j^{3}},
\]
showing that the sum on the right hand side of (\ref{cj-expo-ineq}) is
bounded. Therefore we can pass to the limit $\delta\rightarrow0$ in
(\ref{cj-expo-ineq}). Finally, we let $t\rightarrow\infty$ to finish the proof
of lemma.\bigskip
\end{proof}

After all the preparatory lemmas, the proof of Theorem \ref{expo-conv} now
becomes clear.

\begin{proof}
\textit{(of} \textit{Theorem \ref{expo-conv}) }By Proposition \ref{no-detB1},
for any solution $c_{j}(t)$ of (\ref{0-infode})-(\ref{infIC}) the time
derivative of each cluster density $c_{j}(t)~$approaches to zero showing that
the infinite time limit is an equilibrium $c_{j}^{e}$. But, the equilibrium is
a trivial solution of (\ref{0-infode})-(\ref{infIC}) having the same mass with
original time dependent distribution. By Lemma \ref{L-expo-conv} any two
solutions of (\ref{0-infode})-(\ref{infIC}) approach each other exponentially
fast in time. It follows that the solution $c(t)$ of (\ref{0-infode}%
)-(\ref{infIC}) converges to the $c^{e}$ exponentially fast as in
(\textit{\ref{tot_num-conv}}). Clearly any other solution with the same
initial mass also converges to the same equilibrium. Furthermore, the
equilibrium is unique. This is because if there was any other equilibrium
$d^{e},$ going through the algebra of Lemma \ref{L-expo-conv} for the
nonlinear equations $C^{e}$ and $D^{e}$, we would obtain%
\[
0\leq-(\tilde{a}-8L^{2}\varepsilon)\sum_{j=1}^{\infty}\left\vert C_{j}%
^{e}-D_{j}^{e}\right\vert
\]
from which we would conclude $C_{j}^{e}=D_{j}^{e}$ implying $c_{j}^{e}%
=d_{j}^{e}.$
\end{proof}

\subsection{Exponentially Fast Strong Convergence to Equilibrium}

Theorem \ref{expo-conv} relied heavily on the monotonicity properties of
$a_{j},b_{j}$ functions. It is desirable to relax these conditions. In our
next result, we show that when the total mass is sufficiently small, the
monotonicity assumption can be dropped and it can be shown that convergence to
the equilibrium is exponentially fast in the mass norm. More precisely we want
to prove the following.

\begin{theorem}
\label{expo-conv2}Consider the (\ref{0-infode})-(\ref{infIC}) system. Let the
hypothesis of hypothesis of Theorem 1 be satisfied with a given mass $\rho.$
Let the hypotheses H3 and H4 hold. Assume further that the mass of the system
is sufficiently small. Then the solutions of (\ref{0-infode})-(\ref{infIC})
converge to a unique equilibrium in the sense that
\[
\sum_{j\geq1}j\left\vert c_{j}(t)-c_{j}^{e}\right\vert \leq2\rho e^{-\gamma
t}.
\]

\end{theorem}

The mild growth conditions on the kernels stated in the theorem are as
follows.%
\[
(H4)\text{ \ \ }%
\begin{array}
[c]{c}%
a_{\min}\leq a_{j}\leq\bar{a}j\text{ and }b_{\min}\leq b_{j}\leq\bar{b}j\text{
for }j\geq1\\
\alpha_{\min}\leq\alpha_{j}\leq\bar{\alpha}j\text{ and }\beta_{\min}\leq
\beta_{j}\leq\bar{\beta}j\text{ for }j\geq1
\end{array}
\]
We first need a lemma showing the boundedness of the moments of solutions. As
in the previous subsection, we do not assume detailed balance and so we cannot
use the recursive relation as in Section 3. Also, differently from Section
4.1, due to the faster growth rate in the $a_{j}$ functions, we cannot, in
general, show finiteness of all moments for small mass uniformly. However,
with a modification of Lemma \ref{L-fin-mom}, we can show that for small mass,
the second moment is bounded.

\begin{lemma}
\label{M2-rho}\bigskip Let, for the system (\ref{0-infode})-(\ref{infIC}), the
conditions of Theorem 1 be satisfied and assume that the hypotheses $H3$ and
H4 hold. Then for small enough mass $\rho,$ the system has bounded second moment.
\end{lemma}

\begin{proof}
We show this by formal computations $c_{j}$. It can be made rigorous by
truncated solutions in just the same way as in Section 4.1. Setting
$g_{j}=j^{2}$ in Lemma 1 we get%
\begin{align*}
\sum_{j\geq1}j^{2}\dot{c}_{j}  &  =\sum_{j\geq0}\left(  (j+1)^{2}%
-j^{2}\right)  (a_{j}\rho+B(c)+\varepsilon\alpha_{j}\tilde{B}(c))c_{j}\\
&  +\sum_{j\geq1}\left(  (j-1)^{2}-j^{2}\right)  (jA(c)+b_{j}+\varepsilon
\beta_{j}\tilde{A}(c))c_{j}%
\end{align*}%
\begin{align*}
&  =\sum_{j\geq0}(2j+1)(a_{j}\rho+B(c)+\varepsilon\alpha_{j}\tilde{B}%
(c))c_{j}+\sum_{j\geq1}(-2j+1)(jA(c)+b_{j}+\varepsilon\beta_{j}\tilde
{A}(c))c_{j}\\
&  \leq2\bar{a}\rho\sum_{j\geq0}j^{2}c_{j}+\rho(a_{0}+2\bar{a}\rho)+2\rho
B(c)+B(c)+2\varepsilon\bar{\alpha}\bar{\beta}\rho\sum_{j\geq0}j^{2}%
c_{j}+\varepsilon(\alpha_{0}+\bar{\alpha}\rho)\bar{\beta}\rho\\
&  -2\tilde{a}\sum_{j\geq1}j^{2}c_{j}-2\sum_{j\geq1}jb_{\min}c_{j}%
-2\varepsilon\sum_{j\geq1}j\beta_{\min}c_{j}\tilde{A}(c)+\rho
A(c)+B(c)+\varepsilon\bar{\beta}\rho\tilde{A}(c),
\end{align*}
where, in the fourth and fifth lines, we used $A(c)=\sum_{j\geq0}a_{j}%
c_{j}\leq a_{0}+\sum_{j\geq1}a_{j}c_{j}\leq a_{0}+\bar{a}\rho$ and
$B(c)=\sum_{j\geq1}b_{j}c_{j}\leq\bar{b}\rho\,$\ (similarly for $\tilde{A}(c)$
and $\tilde{B}(c)).$ After rearranging the terms we have%
\[
\sum_{j\geq1}j^{2}\dot{c}_{j}\leq2(\bar{a}\rho+2\varepsilon\bar{\alpha}%
\bar{\beta}\rho-\tilde{a})\sum_{j\geq1}j^{2}c_{j}+\rho\left(  2a_{0}+3\bar
{a}\rho+2\bar{b}\rho+2\bar{b}+2\varepsilon\bar{\beta}(\alpha_{0}+\bar{\alpha
})\rho-b_{\min}\right)  .
\]
If $\rho<\frac{\tilde{a}}{\bar{a}+2\varepsilon\bar{\alpha}\bar{\beta}}$ then
the differential inequality yields that the second moment is bounded and in
particular
\[
M_{2}(\infty)\leq\rho\frac{\left(  2a_{0}+3\bar{a}\rho+2\bar{b}\rho+2\bar
{b}+2\varepsilon\bar{\beta}(\alpha_{0}+\bar{\alpha})\rho-b_{\min}\right)
}{2(\bar{a}\rho+2\varepsilon\bar{\alpha}\bar{\beta}\rho-\tilde{a})}.
\]

\end{proof}

\textbf{Remark.} It is worth saying that there is nothing special about the
$M_{2}$. The proof can be extended, for a given $p>0,$ i.e., $M_{p}%
(\infty)<C\,\ $so long the total mass $\rho$ is small enough. However, the
smallness requirement will depend on the value of $p$.

Next, as in the previous section we show the contraction property of solutions.

\begin{lemma}
\label{expo-conv-rho}\bigskip Let the conditions of Theorem 1 be satisfied and
$c_{j}$ and $d_{j}$ be two solutions of the system (\ref{0-infode}%
)-(\ref{infIC}) system with the same initial mass. Assume that the hypotheses
$H3$ and H4 hold with the total mass (density) $\rho$ and $\varepsilon$ are
small enough. Then the two solutions approach to each other in the sense that
\[
\sum_{j\geq1}j\left\vert c_{j}(t)-d_{j}(t)\right\vert \leq2\rho e^{-\gamma
t}.
\]

\end{lemma}

The general idea of proof is similar to the contraction result in Section 4.1.
However, in this case it is more convenient to use difference of individual
cluster densities (not the tail) to measure the difference of time dependent
solutions, i.e., the function $\xi_{1}(t)=\sum_{j\geq1}j\left\vert
c_{j}(t)-d_{j}(t)\right\vert $. The goal is to show that its derivative
satisfies a differential inequality which yields the result.

\begin{proof}
Let $c_{j}$ and $d_{j}$ and be the time dependent and equilibrium solutions
and $e_{j}=c_{j}-d_{j}$ be the difference. Setting $g_{j}=jsgn(e_{j})$ in
Lemma 1 (with $N\rightarrow\infty)$ one gets%
\begin{align*}
\sum_{j\geq1}jsgn(e_{j})\dot{c}_{j}  &  =\sum_{j\geq0}\left(  (j+1)sgn(e_{j+1}%
)-jsgn(e_{j})\right)  (a_{j}\rho+B(c)+\varepsilon\alpha_{j}\tilde{B}%
(c))c_{j}\\
&  +\sum_{j\geq1}\left(  (j-1)sgn(e_{j-1})-jsgn(e_{j})\right)  (jA(c)+b_{j}%
+\varepsilon\beta_{j}\tilde{A}(c))c_{j},
\end{align*}
Subtracting from above the equation for $\sum_{j\geq1}jsgn(e_{j})\dot{d}_{j}$
and noting $A(c)c_{j}-A(d)d_{j}=e_{j}A(c)+d_{j}A(e)$ and $B(c)c_{j}%
-B(d_{j})d_{j}=e_{j}B(c)+d_{j}B(e)$ (and using similar notations for
$\tilde{A}$ and $\tilde{B})$ we get%
\begin{align}
\sum_{j\geq1}j\left\vert \dot{e}_{j}\right\vert  &  =\sum_{j\geq0}\left(
(j+1)sgn(e_{j+1})-jsgn(e_{j})\right)  (\rho a_{j}e_{j}+e_{j}B(c)+d_{j}%
B(e))\label{norm-dif-rho}\\
&  +\sum_{j\geq1}\left(  (j-1)sgn(e_{j-1})-jsgn(e_{j})\right)  (je_{j}%
A(c)+jd_{j}A(e)+b_{j}e_{j})\nonumber\\
&  +\varepsilon\sum_{j\geq0}\left(  (j+1)sgn(e_{j+1})-jsgn(e_{j})\right)
\alpha_{j}(e_{j}\tilde{B}(c)+d_{j}\tilde{B}(e)),\\
&  +\varepsilon\sum_{j\geq1}\left(  (j-1)sgn(e_{j-1})-jsgn(e_{j})\right)
\beta_{j}(e_{j}\tilde{A}(c)+d_{j}\tilde{A}(e)),
\end{align}
where in (\ref{norm-dif-rho}) we implicitly used $\rho=\sum_{j\geq0}%
jc_{j}=\sum_{j\geq0}jd_{j}$ and $1=\sum_{j\geq0}c_{j}=\sum_{j\geq0}d_{j}$.
Upon distributing the $(j\pm1)sgn(e_{j\pm1})-jsgn(e_{j})$ terms over the terms
inside the parenthesis on the right hand side of (\ref{norm-dif-rho}), in each
line, we produce a total of $10$ terms which we denote by $S_{1},...,S_{10}.$
For each $S$ term we obtain an inequality. For $S_{1},$ using $\left\vert
sgn(e_{j+1})\right\vert \leq1$ and $sgn(e_{j})e_{j}=\left\vert e_{j}%
\right\vert ,$ we write%
\begin{align*}
S_{1}  &  =\sum_{j\geq0}\left(  (j+1)sgn(e_{j+1})-jsgn(e_{j})\right)  \rho
a_{j}e_{j}\leq\rho\sum_{j\geq0}a_{j}\left\vert e_{j}\right\vert \\
&  \leq\rho\sum_{j\geq1}(a_{j}+a_{0})\left\vert e_{j}\right\vert ,
\end{align*}
where in the second line we used $\sum_{j\geq0}e_{j}=0$ (conservation of
volume) giving $\left\vert e_{0}\right\vert \leq\sum_{j\geq1}\left\vert
e_{j}\right\vert .$ Similarly, for $S_{2},$ one has%
\[
S_{2}=\sum_{j\geq0}\left(  (j+1)sgn(e_{j+1})-jsgn(e_{j})\right)  e_{j}%
B(c)\leq2B(c)\sum_{j\geq1}\left\vert e_{j}\right\vert .
\]
For $S_{3}$ we observe $\left\vert B(e)\right\vert \leq\sum_{j\geq1}%
b_{j}\left\vert e_{j}\right\vert $ and obtain%
\begin{align*}
S_{3}  &  =\sum_{j\geq0}\left(  (j+1)sgn(e_{j+1})-jsgn(e_{j})\right)
d_{j}B(e)\\
&  \leq\sum_{j\geq0}(2j+1)d_{j}\sum_{k\geq1}b_{k}\left\vert e_{k}\right\vert .
\end{align*}
For $S_{4}$ term we again use $\left\vert sgn(e_{j-1})\right\vert \leq1$ and
find
\begin{align*}
S_{4}  &  =\sum_{j\geq1}\left(  (j-1)sgn(e_{j-1})-jsgn(e_{j})\right)
je_{j}A(c)\\
&  \leq-\sum_{j\geq1}j\left\vert e_{j}\right\vert A(c).
\end{align*}
The $S_{5}$ term, using $\left\vert A(e)\right\vert \leq\sum_{j\geq0}%
a_{j}\left\vert e_{j}\right\vert \leq\sum_{j\geq1}(a_{j}+a_{0})\left\vert
e_{j}\right\vert ,$ gives%
\begin{align*}
S_{5}  &  =\sum_{j\geq1}\left(  (j-1)sgn(e_{j-1})-jsgn(e_{j})\right)
jd_{j}A(e)\\
&  \leq\sum_{j\geq1}(2j-1)jd_{j}\sum_{k\geq1}(a_{k}+a_{0})\left\vert
e_{k}\right\vert .
\end{align*}
And, $S_{6}$ reads%
\[
S_{6}=\sum_{j\geq1}\left(  (j-1)sgn(e_{j-1})-jsgn(e_{j})\right)  b_{j}%
e_{j}\leq-\sum_{j\geq1}b_{j}\left\vert e_{j}\right\vert .
\]

Looking at the terms one notices that $S_{6}$ cancels part of the term on the
right hand side of $S_{3}$ since $\sum_{j\geq0}d_{j}=1$ which leaves
$\sum_{j\geq0}(2j)d_{j}\sum_{k\geq1}b_{k}\left\vert e_{k}\right\vert .$
Similarly, $S_{1}$ cancels the negative part on the right hand side of $S_{5}$
since $\sum_{j\geq1}jd_{j}=\rho.$ Combining with the rest of the terms in
(\ref{norm-dif-rho}) we get%
\begin{align}
\sum_{j\geq1}j\left\vert \dot{e}_{j}\right\vert  &  \leq2B(c)\sum_{j\geq
1}\left\vert e_{j}\right\vert +\sum_{j\geq0}(2j)d_{j}\sum_{k\geq1}%
b_{k}\left\vert e_{k}\right\vert \label{norm-dif-der2}\\
&  -\sum_{j\geq1}j\left\vert e_{j}\right\vert A(c)+\sum_{j\geq1}(2j)jd_{j}%
\sum_{k\geq1}(a_{k}+a_{0})\left\vert e_{k}\right\vert \nonumber\\
&  +S_{7}+S_{8}+S_{9}+S_{10}.
\end{align}

We now estimate the perturbation terms $S_{7},...,S_{10}$ in a similar
fashion. $S_{7}$ and $S_{8}$ are given by%
\[
S_{7}=\varepsilon\sum_{j\geq0}\left(  (j+1)sgn(e_{j+1})-jsgn(e_{j})\right)
\alpha_{j}e_{j}\tilde{B}(c)\leq\varepsilon\sum_{j\geq1}(\alpha_{0}+\bar
{\alpha}j)\left\vert e_{j}\right\vert \tilde{B}(c),
\]%
\begin{align*}
S_{8}  &  =\varepsilon\sum_{j\geq0}\left(  (j+1)sgn(e_{j+1})-jsgn(e_{j}%
)\right)  \alpha_{j}d_{j}\tilde{B}(e)\leq\varepsilon\sum_{j\geq0}\left(
2j+1\right)  \alpha_{j}d_{j}\left\vert \tilde{B}(e)\right\vert \\
&  \leq\varepsilon\sum_{j\geq0}(2j^{2}\bar{\alpha}+\bar{\alpha}j+\alpha
_{0})d_{j}\sum_{k\geq1}\beta_{k}\left\vert e_{k}\right\vert ,
\end{align*}
where we used $\sum_{j\geq0}\alpha_{j}\left\vert e_{j}\right\vert \leq
\sum_{j\geq1}(\alpha_{j}+\alpha_{0})\left\vert e_{j}\right\vert $ (since
$\left\vert e_{0}\right\vert \leq\sum_{j\geq1}\left\vert e_{j}\right\vert ).$
Similarly, $S9$ and $S10$ are given by%
\[
S_{9}=\varepsilon\sum_{j\geq1}\left(  (j-1)sgn(e_{j-1})-jsgn(e_{j})\right)
\beta_{j}e_{j}\tilde{A}(c)\leq-\varepsilon\sum_{j\geq1}\beta_{\min}\left\vert
e_{j}\right\vert \tilde{A}(c),
\]%
\[
S_{10}=\varepsilon\sum_{j\geq1}\left(  (j-1)sgn(e_{j-1})-jsgn(e_{j})\right)
\beta_{j}d_{j}\tilde{A}(e)\leq\varepsilon\sum_{j\geq1}(2j-1)d_{j}\beta_{j}%
\sum_{k\geq0}\alpha_{k}\left\vert e_{k}\right\vert .
\]
By the bounds given in the theorem $B(c)\leq\sum_{j\geq1}\bar{b}jc_{j}\leq
\bar{b}\rho$ and $A(c)\geq\sum_{j\geq1}\tilde{a}c_{j}\geq\tilde{a}.$ Also,
$\sum_{j\geq1}\left\vert e_{j}\right\vert \leq\sum_{j\geq1}j\left\vert
e_{j}\right\vert $ and $\sum_{k\geq1}(a_{k}+a_{0})\left\vert e_{k}\right\vert
\leq(a_{0}+\bar{a})\sum_{k\geq1}k\left\vert e_{k}\right\vert .$ Then using
$\left\vert e_{0}\right\vert \leq\sum_{j\geq1}\left\vert e_{j}\right\vert $
several times (\ref{norm-dif-der2}) reduces to
\begin{align*}
\sum_{j\geq1}j\left\vert \dot{e}_{j}\right\vert  &  \leq2\bar{b}\rho
\sum_{j\geq1}j\left\vert e_{j}\right\vert +2\rho\bar{b}\sum_{j\geq
1}j\left\vert e_{j}\right\vert -a_{\min}\sum_{j\geq1}j\left\vert
e_{j}\right\vert +2M_{2}(a_{0}+\bar{a})\sum_{k\geq1}k\left\vert e_{k}%
\right\vert \\
&  +\varepsilon\left(  (\alpha_{0}+\bar{\alpha})\rho\bar{\beta}+2\bar{\alpha
}\bar{\beta}M_{2}+\alpha_{0}\bar{\beta}d_{0}+\bar{\alpha}\bar{\beta}%
\rho\right)  \sum_{k\geq1}k\left\vert e_{k}\right\vert +\varepsilon2\bar
{\beta}M_{2}(\alpha_{0}+\bar{\alpha})\sum_{k\geq1}k\left\vert e_{k}\right\vert
.
\end{align*}
From Lemma \ref{M2-rho} we know that $M_{2}=\sum_{j\geq1}j^{2}d_{j}\leq
N(\rho):=\frac{\rho\left(  2a_{0}+3\bar{a}\rho+2\bar{b}\rho+2\bar
{b}+2\varepsilon\bar{\beta}(\alpha_{0}+\bar{\alpha})\rho-b_{\min}\right)
}{2(\bar{a}\rho+2\varepsilon\bar{\alpha}\bar{\beta}\rho-\tilde{a})}$. Hence
one gets the differential inequality%
\begin{equation}
\sum_{j\geq1}j\left\vert \dot{e}_{j}\right\vert \leq\left(  4\bar{b}%
\rho+2(a_{0}+\bar{a})N(\rho)+\varepsilon\left(  (\alpha_{0}+2\bar{\alpha}%
)\rho\beta+2(2\bar{\alpha}+\alpha_{0})\bar{\beta}M_{2}+\alpha_{0}\bar{\beta
}d_{0}+\bar{\alpha}\bar{\beta}\rho\right)  -a_{\min}\right)  \sum_{j\geq
1}j\left\vert e_{j}\right\vert . \label{wght-norm-der}%
\end{equation}
It is then clear that, for $\rho$ and $\varepsilon$ small enough, the
parenthesis on the right hand side of (\ref{wght-norm-der}) has a negative
value (say $-\gamma<0)$ giving
\[
\sum_{j\geq1}j\left\vert \dot{e}_{j}\right\vert \leq\sum_{j\geq1}j\left\vert
e_{j}(0)\right\vert e^{-\gamma t}\leq\sum_{j\geq1}j(c_{j}(0)+d_{j}%
(0))e^{-\gamma t}\leq2\rho e^{-\gamma t}%
\]
which proves the lemma.
\end{proof}

As the last ingredient for the theorem, we have the existence of the
equilibrium solutions which is analogous to Proposition \ref{no-detB1}.

\begin{proposition}
\label{no-detB2}Assume the conditions of Theorem 1. Let hypotheses H3, H4 hold
and the total mass be small enough. Then for any solution satisfying
(\ref{0-infode})-(\ref{infIC})\ one has%
\[
\lim_{t\rightarrow\infty}\left\vert \dot{c}_{j}(t)\right\vert =0\text{
\ \ }(\text{for }j\geq0)
\]

\end{proposition}

The proof follows similar steps to Proposition \ref{no-detB1}, hence we skip
it. Collecting all of the results we can now prove the main theorem of this subsection.

\begin{proof}
\textit{(of Theorem \ref{expo-conv2})}. For mass sufficiently small, by the
contraction property, any two solutions approach each other exponentially fast
in the sense of Lemma \ref{expo-conv-rho}. By Lemma \ref{no-detB2} each
solution goes to the equilibrium solution in the large time. Hence all time
dependent solutions with equal mass converge to the same equilibrium solution.
\end{proof}

\section{CONCLUSION}

In this article, we studied the large time behavior of the EDG\ system,
particularly the convergence of solutions to equilibrium with explicit
convergence rates where possible. Due to the complexities arising in a fully
general kernel we focussed on two special but fairly general classes of
separable kernels (in product and sum forms).

For the first class of kernels $K(j,k)=b_{j}a_{k},$ we showed the existence of
equilibria under the assumption of detailed balance. The crucial finding is
that not all initial mass values can support equilibrium solution. Much like
in the Becker-Doring system, above a critical mass $\rho_{c},$ the EDG system
undergoes a dynamic phase transition. By employing well known method of
entropy functionals we showed the strong convergence of solutions to the
equilibrium for initial masses below the critical mass and weak convergence to
the critical equilibrium density for initial masses above the critical mass.
The question of how fast these convergences occur in each case is left open
for future study.

For the second class of kernels given by $K(j,k)=ja_{k}+b+\varepsilon\beta
_{j}\alpha_{k}$, we showed the existence of equilibrium as a by-product of a
contraction property which followed from the monotonicity of $b_{j},a_{j}$ an
assumption motivated by the heuristic interpretations that $a_{j,}b_{j}$
represent the \textit{import/growth} and \textit{export/fragmentation}. While
these analogies (between the $a_{j},b_{j}$ of BD\ systems and EDG systems) are
appealing and acceptable to a certain extend, one should bear in mind that, in
the exchange systems the dynamics is so intertwined that $a_{j},b_{j}$ should
not be regarded too simplistically or being mere copies of coagulation and
fragmentation rates as in the BD system. Nevertheless, the arguments suggest
that the result should generalize which we state as a conjecture

\textbf{Conjecture.} \textit{Consider the EDG\ system (\ref{0-infode}%
)-(\ref{infIC}) system. Let the hypothesis of hypothesis of Theorem 1 be
satisfied. Assume for the kernel }$K(j,k)$\textit{ that it is non-decreasing
in the first component and non-increasing in the second component. Then the
solutions of (\ref{0-infode})-(\ref{infIC}) converge to equilibrium
exponentially fast in the sense of Theorem \ref{expo-conv}.}

For second class of kernels it is also shown that the monotonicity assumption
can be replaced with a bound condition on the total mass of the system in
which case one can show exponential convergence in the strong norm. We do not
know, if this requirement is only a technical assumption or a requirement as
in the case of first class $I$ kernels.

\textbf{Acknowledgement.} The first author thanks Colm Connaughton for many
fruitful discussions. He also thanks Miguel Escobedo for many useful comments
and hospitality during his one-month visit of University Pais Vasco in
November 2018. The first author is indebted to the hospitality of Hausdorff
Centre for Mathematics where this work was initiated March 2018. He also
thanks Andr\'{e} Schlichting for useful references where he came into contact
during a workshop at the University of Warwick in September 2018.

This work is supported by the Marie Curie Fellowship of European Commission,
grant agreement number: 705033. Also, the authors acknowledge support from the
Hausdorff Centre for Mathematics and CRC 1060 on Mathematics of emergent
effects from University of Bonn.

\end{document}